\documentclass[a4paper,oneside,english,reqno]{amsart}

\usepackage[utf8]{inputenc}
\usepackage[T1]{fontenc}

\DeclareFontFamily{OMX}{mlmex}{}
\DeclareFontShape{OMX}{mlmex}{m}{n}{%
   <->mlmex10%
   }{}%
\usepackage{mlmodern}

\usepackage[hscale=0.6, vscale=0.7]{geometry}
\usepackage{babel}
\usepackage[numbers]{natbib}
\usepackage{hyperref}
\hypersetup{%
  colorlinks=true,%
  citecolor=[RGB]{120,29,126},%
  pdfauthor={Jean-François Burnol},%
  pdfsubject={Euler's constant},%
  pdfstartview=FitH,%
  pdfpagemode=UseNone,%
}

\usepackage{amsmath,amsthm,amssymb}
\usepackage{mathtools}
\DeclareMathOperator{\Log}{Log}
\DeclareMathOperator{\Arg}{Arg}
\DeclarePairedDelimiterX\Iffint[2]{\lbrack\!\lbrack}{\rbrack\!\rbrack}{#1,#2}
\DeclarePairedDelimiterX\Ioo[2]{\lparen}{\rparen}{#1,#2}
\DeclarePairedDelimiterX\Iof[2]{\lparen}{\rbrack}{#1,#2}
\DeclarePairedDelimiterX\Ifo[2]{\lbrack}{\rparen}{#1,#2}
\DeclarePairedDelimiterX\Iff[2]{\lbrack}{\rbrack}{#1,#2}

\newcommand\NN{\mathbb{N}}
\newcommand\ZZ{\mathbb{Z}}

\newcommand\CC{\mathbb{C}}
\newcommand\cA{\mathcal{A}}
\newcommand\cB{\mathcal{B}}
\newcommand\cC{\mathcal{C}}

\newcommand\e{\mathsf{e}}
\newcommand\dt{\mathrm{d}t}

\newcommand\dx{\mathrm{d}x}

\newcommand\dz{\mathrm{d}z}

\newcommand\dds{\frac{\mathrm{d}}{\mathrm{d}s}}

\newcommand\ddm{\frac{\mathrm{d}}{\mathrm{d}m}}
\newcommand\ddz{\frac{\mathrm{d}}{\mathrm{d}z}}

\DeclareMathOperator\Res{Res}
\theoremstyle{plain}
\newtheorem{theo}{Theorem}
\newtheorem{prop}{Proposition}
\newtheorem{lem}{Lemma}

\theoremstyle{definition}

\newtheorem{rema}{Remark}

\allowdisplaybreaks

\title[Some series for gamma]{%
  Some geometric series for Euler's constant}

\author[J.-F. Burnol]{Jean-François Burnol}

\newcommand\arxivurl[1]{\href{https://arxiv.org/abs/#1}{\textsf{arXiv:#1}}}

\date{June 14, 2026}



\usepackage[np,autolanguage]{numprint}
\AtBeginDocument{\npthousandsep{\allowbreak\,}}

\begin{document}

\begin{abstract}
  We provide representations of Euler's constant $\gamma=0.577...$ as series
  which converge geometrically fast (but use a certain sequence whose
  computation induces a quadratic cost).  The asymptotic oscillations of these
  coefficients are determined to all orders.  A result of independent interest,
  about sufficient conditions for the validity, in the case of unbounded
  parameters, for the Tricomi-Erdélyi asymptotic expansion of the ratio of two
  Gamma functions, is established for that purpose.
\end{abstract}

\maketitle

\section{Main results}

Let $e_0=0$, and define positive rational numbers $e_m$ for $m\geq1$ by this
recurrence:
\begin{equation}
  \label{eq:emrec}
  e_{m} = \frac{2^{m+1} + \sum_{j=1}^m \binom{m+1}{j} e_{m-j}}{2^{m+1} -2}.
\end{equation}
This gives the sequence
$0$,
$2$,
$\frac{7}{3}$,
$\frac{8}{3}$,
$\frac{133}{45}$,
$\frac{16}{5}$,
\dots{} ,
$e_{10} = \frac{163287}{40579}$,
\dots{}. See Table~\ref{table:3} for more values.
\begin{table}[htbp]
\def\arraystretch{1.5}
\caption{The coefficients $e_m$, $1\leq m\leq 20$.\\ See also
  \url{https://oeis.org/A372422}.}
\label{table:3}
  \begin{tabular}[t]{cc}
    $m$& $e_m$\\\hline\noalign{\vskip2\jot}
    $1$ &$2$
    \\[3\jot]
    $2$ &$\displaystyle\frac{7}{3}$
    \\[3\jot]
    $3$ &$\displaystyle\frac{8}{3}$
    \\[3\jot]
    $4$ &$\displaystyle\frac{133}{45}$
    \\[3\jot]
    $5$ &$\displaystyle\frac{16}{5}$
    \\[3\jot]
    $6$ &$\displaystyle\frac{3221}{945}$
    \\[3\jot]
    $7$ &$\displaystyle\frac{3392}{945}$
    \\[3\jot]
    $8$ &$\displaystyle\frac{100391}{26775}$
    \\[3\jot]
    $9$ &$\displaystyle\frac{20848}{5355}$
    \\[3\jot]
    $10$ &$\displaystyle\frac{163287}{40579}$
    \\[3\jot]
\hline
  \end{tabular}
\qquad\qquad
  \begin{tabular}[t]{cc}
    $m$& $e_m$\\\hline\noalign{\vskip2\jot}
    $11$ &$\displaystyle\frac{7567072}{1826055}$
    \\[3\jot]
    $12$ &$\displaystyle\frac{10605587147}{2492565075}$
    \\[3\jot]
    $13$ &$\displaystyle\frac{1551804656}{356080725}$
    \\[3\jot]
    $14$ &$\displaystyle\frac{1732332761353}{388911367845}$
    \\[3\jot]
    $15$ &$\displaystyle\frac{252492267136}{55558766835}$
    \\[3\jot]
    $16$ &$\displaystyle\frac{2313623814645529}{499751107680825}$
    \\[3\jot]
    $17$ &$\displaystyle\frac{261522788700176}{55527900853425}$
    \\[3\jot]
    $18$ &$\displaystyle\frac{69661896931499841923}{14556250513419389775}$
    \\[3\jot]
    $19$ &$\displaystyle\frac{2828470111061381408}{582250020536775591}$
    \\[3\jot]
    $20$ &$\displaystyle\frac{23101294621895391907711}{4689192129680103420375}$
    \\[3\jot]
\hline
  \end{tabular}
\end{table}
They can be expressed, as we will see, in terms of Bernoulli numbers (using
the convention $B_1=-\frac12$):
\begin{equation}
  \label{eq:ember}
  e_m = -\sum_{k=1}^m \binom{m+1}{k}B_k\frac{2^k}{2^k-1} 
= 1 - \sum_{k=1}^m \binom{m+1}{k}B_k\frac{1}{2^k-1}\,.
\end{equation}
(the second formulation supposes $m\geq1$).

Let $\gamma$ be as usual the Euler-Mascheroni constant.  Here is the main Theorem:
\begin{theo}\label{thm:main}
  Let $\ell\geq2$. Let $(e_m)$ be the sequence defined by Equation~\eqref{eq:emrec}. There holds
  \begin{equation}\label{eq:main}
    \gamma = \sum_{n=1}^{2^{\ell-1}-1}\frac1n - (\ell-1)\log 2  +
    \sum_{m=1}^\infty\frac{(-1)^{m-1}e_m}{m+1}
                    \sum_{2^{\ell-1}\leq n<2^{\ell}} \frac1{n^{m+1}}.
  \end{equation}
In particular for $\ell=2$:
  \begin{gather*}
    \gamma = 1 - \log 2 +
    \sum_{m=1}^\infty\frac{(-1)^{m-1}e_m}{m+1}\Bigl(2^{-m-1}+3^{-m-1}\Bigr),
\\
\shortintertext{and for $\ell=3$:}
    \gamma = 1 + \frac12 + \frac13 -2\log 2 + 
    \sum_{m=1}^\infty\frac{(-1)^{m-1}e_m}{m+1}
                   \Bigl(4^{-m-1}+5^{-m-1}+6^{-m-1}+ 7^{-m-1}\Bigr).
  \end{gather*}
\end{theo}
The \emph{level} $\ell$ is the number of binary digits used by the integers
whose inverse powers are added in Equation~\eqref{eq:main}.  In Table
\ref{table:1}, we illustrate the result of computations well within the reach
of a diligent human calculator: they use each $10$ terms of the series
for $\ell$ respectively equal to $2$, $3$, \dots, $7$, and $20$ terms for $\ell=2, 3, 4$.

\begin{table}[htp]
\def\arraystretch{1.2}
\caption{Table of approximations of $\gamma$ (decimal expansions in the middle column are truncated)}\label{table:1}
  \begin{tabular}{cll}
    $\ell$& Series from Thm.\@~\ref{thm:main} up to $e_m=e_{10}$ & last term\\\hline
    2&\np{0.57715}\dots &${}\approx\np{-0.0001807}$\\
    3&\np{0.577215646}\dots &${}\approx\np{-9.590e-8}$\\
    4&\np{0.5772156648954}\dots&${}\approx\np{-6.001e-11}$\\
    5&\np{0.5772156649015305}\dots&${}\approx\np{-4.481e-14}$\\
    6&\np{0.57721566490153285960}\dots&${}\approx\np{-3.782e-17}$\\
    7&\np{0.57721566490153286060605}\dots&${}\approx\np{-3.425e-20}$\\
\hline\hline
    $\gamma$&\np{0.577 215 664 901 532 860 606 512 090}\dots&\\
\hline\hline
    $\ell$& Series from Thm.\@~\ref{thm:main} up to $e_m=e_{20}$ & last term\\
    4&\np{0.5772156649015328606035}\dots&${}\approx\np{-2.785e-20}$\\
    3&\np{0.577215664901522}\dots&${}\approx\np{-5.384e-14}$\\
    2&\np{0.577215628}\dots &${}\approx\np{-1.119e-7}$\\
\hline
  \end{tabular}
\end{table}

A numerical implementation using \textsf{Python} and the \textsf{mpmath}
library is provided at \url{https://burnolmath.gitlab.io/dyadic-gamma/} (it
uses by default $\ell=8$).  The coefficients $e_m$ are computed using the
recurrence~\eqref{eq:emrec} (the larger the $m$, the less the required
precision), which induces a quadratic cost which is redhibitory for large
scale computations.  Equation~\eqref{eq:ember} is not immediately usable numerically for large $m$, as individual
terms are much larger than $e_m$.  Indeed it turns out that $e_m \sim
\log_2(m)$, more precisely the difference is bounded.  See Table~\ref{table:2}
and Figure~\ref{fig:1}, which displays intriguing oscillations of $e_m -
H_{m+1}/\log(2)$ ($H_n=\sum_{1\leq j\leq n}j^{-1}$) as a function of
$\log_2(m)$.

\begin{table}[htbp]
\def\arraystretch{1.2}
\caption{Comparison of $e_m$ with $(\log 2)^{-1}H_{m+1}$}
\label{table:2}
  \begin{tabular}[t]{cl}
    $m$& \multicolumn{1}{c}{$e_m - (\log 2)^{-1}H_{m+1}$}\\
\hline
    $1$ &\np{-0.164042561333}\dots
\\
    $2$ &\np{-0.311607574963}\dots
\\
    $3$ &\np{-0.338948001852}\dots
\\
    $4$ &\np{-0.338598121140}\dots
\\
    $5$ &\np{-0.334602850177}\dots
\\
    $6$ &\np{-0.332236533267}\dots
\\
    $7$ &\np{-0.331621032426}\dots
\\
    $8$ &\np{-0.331908031990}\dots
\\
    $9$ &\np{-0.332424034678}\dots
\\
    $10$ &\np{-0.332833436510}\dots
\\
\hline
  \end{tabular}
\qquad
  \begin{tabular}[t]{cl}
    $m$& \multicolumn{1}{c}{$e_m - (\log 2)^{-1}H_{m+1}$}\\
\hline
    $11$ &\np{-0.333041375341}\dots
\\
    $12$ &\np{-0.333074427538}\dots
\\
    $13$ &\np{-0.332999205051}\dots
\\
    $14$ &\np{-0.332880028058}\dots
\\
    $15$ &\np{-0.332762836214}\dots
\\
    $16$ &\np{-0.332673218057}\dots
\\
    $17$ &\np{-0.332620451671}\dots
\\
    $18$ &\np{-0.332603053818}\dots
\\
    $19$ &\np{-0.332613748636}\dots
\\
    $20$ &\np{-0.332643144549}\dots
\\
\hline
  \end{tabular}
\end{table}
\begin{figure}[htbp]
   \caption{$e_m - (\log 2)^{-1}H_{m+1}$ as function of $\log_2(m)$, $16\leq m \leq 500$. Notice the oscillations as function of $\log_2(m)$ and their small amplitudes.}
\label{fig:1}
  \centering
\includegraphics[width=\linewidth]
   {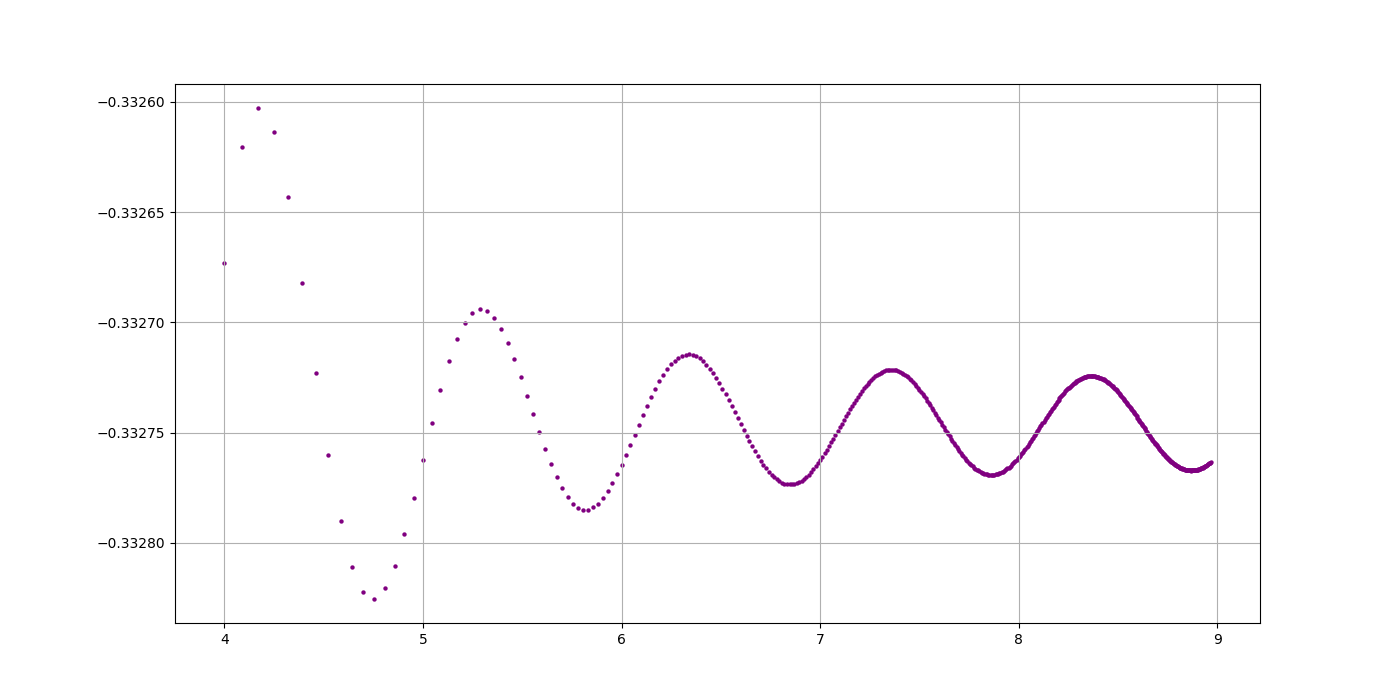}
\end{figure}


This is the time to reveal that the sequence $(e_m)$ (and its peculiar
asymptotic) has long been known in the field of theoretical computer science.
The numerators and denominators have dedicated OEIS pages
\href{https://oeis.org/A372422}{A372422} and
\href{https://oeis.org/A372423}{A372423}.  Paraphrasing their description
in a less precise way, $e_m$ is \emph{the expected depth of trees related to
  the process of recursively randomly eliminating people via coin tosses, from
  a group of initial cardinality $m+1$}.  The quantity $d_N$ of \cite[\S4, Thm.\@
14]{prodinger1993} is $e_{N-1}$.  It is one among a family of related
sequences arising in the theoretical analysis of search algorithms, to which
\textsc{Knuth} dedicated an entire chapter of his influential treatise
(\cite[\S6.3 ``Digital searching'']{knuth}).
We were led to the sequence $(e_m)$ via the
analysis of zeta series with missing digits (inclusive of the case with no
missing digits\dots) which we did in \cite{burnolzeta}, and
the occurrence in the context of \emph{digital search algorithms} is in
retrospect perhaps not so surprising.

In \cite{prodinger1993} \textsc{Prodinger} mentions an asymptotic $d_N
\texttt{"}{\sim}\texttt{"} \log_2(N) + \frac12 - \delta_2(\log_2(N))$ where
$\delta_2$ is a $1$-periodic function with zero average and small amplitude,
given as a Fourier series in $\log_2(N)$ whose coefficients are the values of
$(\log2)^{-1}\Gamma(s)\zeta(s)$ at the roots of $2^s-2=0$,  $s\neq1$.  Such type
of asymptotic occurs on multiple occasions in \textsc{Knuth}'s treatise: many
elements related to their analysis are presented in \cite[\S5.2.2, pp.\@
128-134]{knuth} and in \cite[\S6.3]{knuth}.  But the sequence $(d_N)$
considered by \textsc{Prodinger} presents technical difficulties (some, but
not all, originating in the presence of the Riemann zeta function), and the
presentation in \cite{prodinger1993} is only descriptive.

The matter was mentioned with more details in \cite[p.\@ 113]{flajosedge1995}
by \textsc{Flajolet} and \textsc{Sedgewick} who explain (based on a private
communication by \textsc{Grabner}) why the integral on a vertical line in the
complex plane, used as starting point in \cite{prodinger1993}, does represent
$1-d_N$ (which is denoted $V_N$ in \cite{flajosedge1995}).  How the periodic
function $\delta_2$, which emerges asymptotically from the computed residues,
actually provides an approximation with an $o(1)$ error is not explained: only
an $O(\sqrt N)$ error bound is indicated. An $O(\sqrt N)$ occurred earlier in
\cite{flajosedge1995} (on page 112) when discussing a quantity $U_N$ studied
by \textsc{Knuth} on pages 130 to 133 of \cite{knuth}, which is $\sim
N\log_2(N)$.  Perhaps $O(1/\sqrt{N})$ was intended for $V_N$ which is
$\sim-\log_2(N)$?




Don~\textsc{Knuth} had
already studied in his Treatise \cite{knuth} combinations of Bernoulli numbers
of the type of Equation \eqref{eq:ember}, in particular
$\sum_{2\leq k < n}\binom{n}{k}\frac{B_k}{2^{k-1}-1}$ (let's call it $W_n$),
which occurs on the right-hand side of Eq.\@ (18) of \cite[\S6.3]{knuth}. It is
the subject of \cite[Ex.\@ 6.3-34]{knuth}.  This $W_n$ is not the same as
$e_n$, but $n^{-1}W_n$ is a close relative, and here too there is an
asymptotic with a logarithmic main term ($\frac12\log_2(n/\pi)$), a constant,
an oscillating factor with zero mean, and an error term which is $O(n^{-1})$.

In the present paper, we start by obtaining the analogous for the $(e_m)$
sequence, then we push the analysis further by another method and give the
full asymptotic to all orders in inverse powers of $m$ (which will be
decorated by periodic functions of $\log_2(m)$) and we then extend similarly
the result of \cite[Ex.\@ 6.3-34]{knuth}.
\begin{theo}\label{thm:em}
  The sequence $(e_m)$ obeys the following asymptotic:
  \begin{equation}\label{eq:asym}
e_m = \log_2(m) + \frac12 + \phi(\log_2(m)) + O(m^{-1}),
\end{equation}
where $\phi$ is a $1$-periodic function, which is
analytic for $|\Im t|<\pi/(2\log(2)$), has zero mean, and verifies
\begin{equation}\label{eq:thm2phiseries}
    \phi(t) = -t -\frac12 + \sum_{l=0}^\infty \bigl(1 - \frac{2^{t-l}}{\e^{2^{t-l}}-1}\bigr)
  - \sum_{l=-1}^{-\infty} \frac{2^{t-l}}{\e^{2^{t-l}}-1}\,,
\end{equation}
and
\begin{equation}\label{eq:thm2fourier}
  \phi(t) = - (\log 2)^{-1}
     \sum_{n\in\ZZ,n\neq 0}\Gamma\bigl(1-\frac{2\pi i n}{\log 2}\bigr)
                \zeta\bigl(1-\frac{2\pi i n}{\log 2}\bigr)\e^{2\pi i nt}\,.
\end{equation}
\end{theo}

\begin{figure}[htbp]
   \caption{$e_m - (\log 2)^{-1}H_{m+1}$ as function of $\log_2(m)$, and the
     graph of $-\frac{\gamma}{\log 2}+\frac12+\phi(t)$, $16\leq m \leq 1000$, $t=\log_2(m)$.}
\label{fig:2}
  \centering
\includegraphics[width=\linewidth]
   {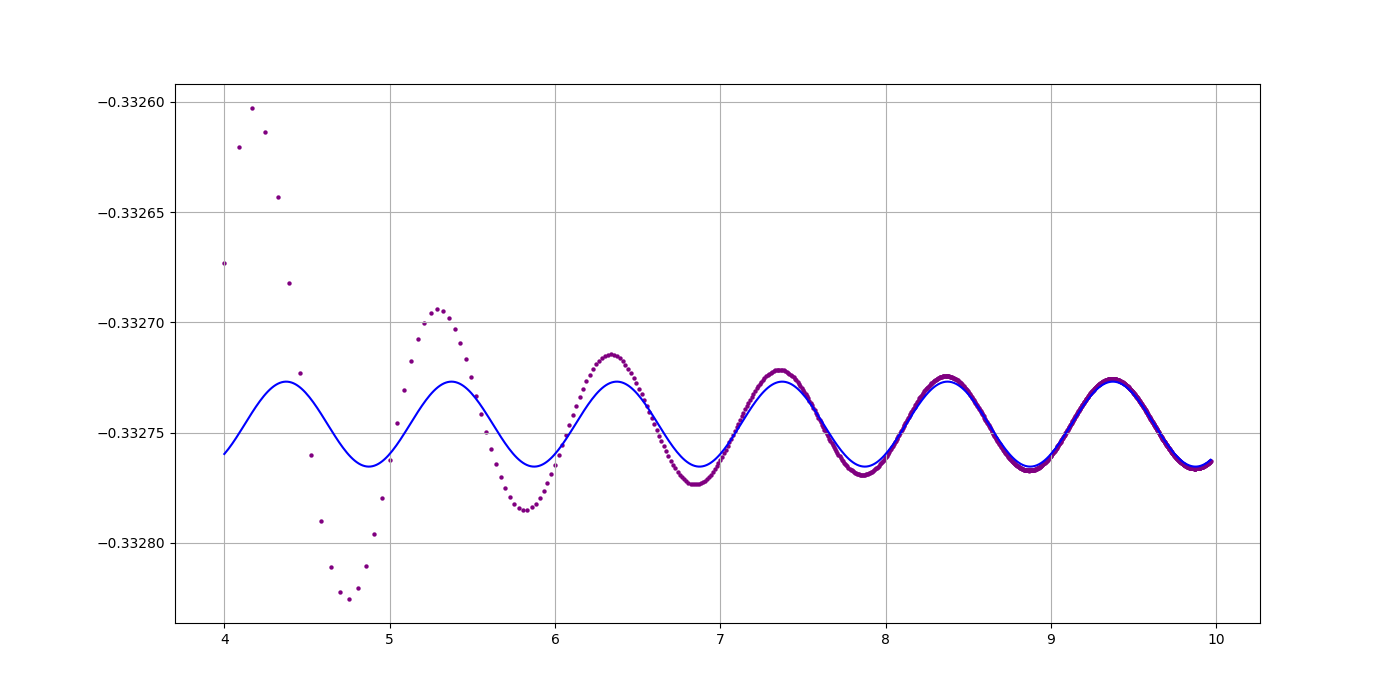}
\end{figure}
\begin{rema}
  Due to the $O(m^{-1})$ error term it does not matter here whether we are
  talking about $e_m$ ($=d_{m+1}$) or about $d_m$ from \cite[Thm.\@
  14]{prodinger1993}, or even about $e_{m-2}$ for example.  To this order,
  they are all the same.
\end{rema}
\begin{rema}\label{rem:psi}
  Numerical computations suggest that the $O(m^{-1})$ error term
  in~\eqref{eq:asym} is $m^{-1}\psi(\log_2(m)) + o(m^{-1})$ with some
  $1$-periodic function $\psi$ averaging to $\frac{3}{2\log(2)}$ (whether we
  use $m$ or $m+1$ now matters for what $\psi$ actually is).  Related to this,
  for $m$ up to a few thousands, $\frac{H_{m+1}}{\log(2)} - \frac{\gamma}{\log
    2} + \frac12$, which differs from $\log_2(m)+\frac12$ by a term equivalent
  to $\frac{3}{2\log(2)m}$, proves to be a much better approximation to $e_m$
  than $\log_2(m) + \frac12$.  These numerical observations are explained by
  the next Theorem.
\end{rema}

For both the $(e_m)$ sequence and the sequence of \cite[Ex.\@ 6.3-34]{knuth}
we prove that the periodic function is but the zeroth term of a complete
asymptotic expansion in inverse powers of $m$, with decorations by periodic
functions of $\log_2(m)$.
\begin{theo}\label{thm:fullasymp}
  Let the polynomials $P_j$ be defined by the conditions $P_0=1$, $P_j(0)=0$
  for $j>0$ and $P_j(t+1)- P_j(t)=-t P_{j-1}(t+1)$, so that $P_1(t) = -
  t(t-1)/2$, $P_2(t) = t(t-1)(t+1)(3t-2)/24$, and $P_3(t)=-t^2(t-1)^2(t+1)(t+2)/48$.  Let, for $k\in\ZZ$, non-zero, $\chi_k = 2\pi i k/\log(2)$.  Let $J$ be a positive integer.  There holds
  \begin{equation}\label{eq:emasymp}
    \log(2) e_m  = H_{m+1} + \frac{\log 2}2 - \gamma 
    \begin{aligned}[t]
      &- \psi_0\bigl(\log_2(m+2)\bigr)
\\
      &- \frac{\psi_1\bigl(\log_2(m+2)\bigr)}{m+2}
\\ &\dots 
\\
      &- \frac{\psi_{J-1}\bigl(\log_2(m+2)\bigr)}{(m+2)^{J-1}} + O\bigl(m^{-J}\bigr)\,,
    \end{aligned}
  \end{equation}
  where $\psi_j$ for $j\geq 0$ is the $1$-periodic function, with zero mean,
  given by the Fourier series
  \begin{equation}
    \psi_j(t) = \sum_{\substack{k\in\ZZ\\k\neq 0}}P_j(-\chi_k)\Gamma(1-\chi_k)\zeta(1-\chi_k)\e^{2\pi i kt}\,.
  \end{equation}
  Let $a$ be some arbitrary real number (for example $a=1$ or
  $2$).  The asymptotic development for $m\to\infty$ in inverse powers of $m+2
  - a$ is (keeping here $H_{m+1}$ exactly represented):
  \begin{equation}\label{eq:asympgen}
    e_m \log(2) \sim H_{m+1} + \frac{\log 2}2 - \gamma - 
     \sum_{j=0}^\infty \frac{\psi_{j,a}
                           \bigl(\log_2(m+2-a)\bigr)}{(m+2-a)^j} 
  \end{equation}
  with
  \begin{equation}
    \psi_{j,a}(t) = \sum_{\substack{k\in\ZZ\\k\neq 0}}
       P_{j,a}(-\chi_k)\Gamma(1-\chi_k)\zeta(1-\chi_k)\e^{2\pi i kt}\,,
  \end{equation}
   where the polynomials $P_{j,a}$ are determined by the conditions
$P_{0,a}=1$, $P_{j,a}(0) = 0$ for $j>0$, and $
    P_{j,a}(t+1) - P_{j,a}(t) = -(t + a)P_{j-1,a}(t+1)$.
\end{theo}
\begin{rema}
  The $\sim$ notation is used in the second part of the Theorem as a shortcut
  for a sequence of statements with finite sums and big-$O$'s as in Equation
  \eqref{eq:emasymp} of its first part.

  To get the complete expansion of the quantity $d_N$ of \cite[Thm.\@
  14]{prodinger1993} in inverse powers of $N$, use the above Theorem with
  $m=N-1$ and $a=1$.  One has actually $P_{j,1}(t) = P_{j,0}(t+1)$, due to
  the phenomenon that $P_{j,0}(1)= 0$ for all $j>0$.  See the start of
  Theorem \ref{thm:fullasymp} for the polynomials $P_{j,0}$, $1\leq j \leq
  3$, hence also, replacing $t$ by $t+1$, for the $P_{j,1}$'s. But no similar
  relation exists with $P_{j,2}(t)$.  One finds $P_{1,2}(t) = -t(t+3)/2$,
  $P_{2,2}(t) = t(t+1)(t+2)(3t+13)/24$, and $P_{3,2}(t) =
  -t(t+1)(t+2)(t+3)(t^2+7t+8)/48$.  See Remark \ref{rem:Pja} for the relation
  of the polynomials $P_{j,a}$ with generalized Bernoulli polynomials.
\end{rema}

A topic of independent interest is that in order to establish Theorem
\ref{thm:fullasymp} we needed to prove the validity of the
\textsc{Tricomi}-\textsc{Erdélyi} asymptotic expansions \cite{tricomi1951} for
$\Gamma(z+\alpha)/\Gamma(z+\beta)$ as $z\to\infty$, without requiring
$\max(|\alpha|, |\beta|)= O(1)$.   For our purposes, the case of $z$ real
positive, or even integer, would have sufficed, but allowing (at least one)
unbounded imaginary parameters is required.

We prove that the condition $\max(|\alpha|^2, |\beta|^2)= O(|z|)$ suffices,
see Theorem \ref{thm:gammaratio} for the precise statement; mind in particular
that if we choose in some manner parameters $\alpha(z)$, $\beta(z)$, we need
in general (but not always)
$\max(|\alpha(z)|^2,|\beta(z)|^2)=o(|z|)$ to claim an \emph{asymptotic
  expansion} (of some sort) after the substitution.
Alternative methods of proof could perhaps start with revisiting how the
asymptotic for fixed $\alpha$ and $\beta$ can be obtained via \emph{Watson's
  Lemma} (cf.\@ \cite[\S1.4]{luke1969volI}, \cite[\S5.1]{olver1997}, \cite[\S
I.5]{wong2001}), with the perspective to check the uniformity under the
condition $\max(|\alpha|^2, |\beta|^2)= O(|z|)$, or examine with the same
uniformity perspective the Stirling asymptotics of $\log \Gamma(z+a)$ as
obtained via the Euler-Maclaurin expansion (\cite[\S4.1]{olver1997}).  We have
chosen an approach somewhat akin to that latter option, but with starting
point the no-parameter Stirling series for $\log\Gamma(z)$ (\cite[\S12.33,
\S13.6]{whittakerwatson4thed}, \cite[\S2.11]{luke1969volI}), and not needing
any explicit formula neither for the coefficients nor the remainder.

  All the literature we could peruse assumes the
  parameters to be fixed or at least bounded. 
  We did not find, e.g.\@ in \cite{erdelyiI}, \cite{abrasteg},
  \cite{luke1969volI} or \cite{olver1997}, the statement we needed, nor in any
  other reference we could access.  For example the \textsc{NIST} Digital
  Library of Mathematical Functions \cite[\S5.11(iii)]{NIST:DLMF} explicitly
  states (\url{https://dlmf.nist.gov/5.11.iii})
that the parameters are constants. As an entry point to the recent
  literature, we mention \cite{yangtianha2020} and the review paper
  \cite{qi2010}.  These references appear to be mainly (even exclusively)
  occupied with the real variable, but they include numerous references and
  discuss the earlier literature. As in \cite{tricomi1951} or
  \cite{fields1966} the parameters appear to be fixed or at least bounded.


\section{Geometric series for \texorpdfstring{$\gamma$}{gamma}}

We obtain Theorem~\ref{thm:main} as a corollary of \cite[Th.\@ 1]{burnolzeta_umosc}.
\begin{proof}[Proof of Theorem \ref{thm:main}]
  We consider the sequence $(c_m(s))$ 
  of holomorphic functions in the right half-plane
  $\Re s>0$ which are defined inductively by $c_0(s)=1$ and, for $m\geq1$:
  \begin{equation}\label{eq:cm}
    c_m(s) = \frac1{2^{m+s}-2}\sum_{j=1}^m \binom{m}{j} c_{m-j}(s)\,.
  \end{equation}
  As we have limited to $\Re(s)>0$, dividing by $2^{m+s}-2$ with $m\geq1$
  introduces no poles.  Using the notations of \cite{burnolzeta_umosc}, we
  have $c_m(s)=(2^s -2)2^{-s}u_m(s)$, where the sequence $(u_m(s))_{m\geq0}$
  verifies the same recurrence, but is initialized with $u_0(s) =
  2^{s}/(2^s-2)$ (which has poles on the line $\Re s = 1$).  In
  \cite{burnolzeta_umosc}, there is more generally an integer parameter $b>1$, which
  in Equation~\eqref{eq:cm} has been set to $b=2$.  There are generalizations
  of Theorem~\ref{thm:main} to a general $b>1$, where the analog of the
  recurrence from Equation~\eqref{eq:emrec} has additional factors given by
  the power sums of the base-$b$ digits (see
  \cite[Eq.\@(5)]{burnolzeta_umosc} for how Equation \eqref{eq:cm} must be
  formulated if using powers of $b>2$).

  Let $\ell\geq2$, the \emph{level}, be some integer greater than one.
  As a corollary to \cite[Th.\@ 1 \& Eq.\ (4)]{burnolzeta_umosc}, the Euler alternating
  series $\eta(s)=\sum_{n=1}^\infty (-1)^{n-1}n^{-s}=(2^s-2)2^{-s}\zeta(s)$
  (here $\Re s>0$) verifies, with local uniform convergence of the series:
  \begin{equation}\label{eq:eta}
    \eta(s) = \frac{2^s-2}{2^s}\sum_{0<n<2^{\ell-1}} \frac1{n^s}
      + \sum_{2^{\ell-1}\leq n<2^{\ell}} \frac1{n^{s}}
      + \sum_{m=1}^\infty (-1)^m \frac{(s)_m}{m!} c_m(s) 
                 \sum_{2^{\ell-1}\leq n<2^{\ell}} \frac1{n^{s+m}}\,.
  \end{equation}
We compute the
  derivative at $s=1$ (here $H_0=0$):
  \begin{equation}\label{eq:der1}
    \eta'(1) = (\log 2) H_{2^{\ell-1}-1} 
        + \sum_{m=0}^\infty(-1)^m\left\{
        \begin{aligned}[c]
          &H_mc_m(1)\sum_{2^{\ell-1}\leq n<2^{\ell}} \frac1{n^{m+1}}\\
          &+c_m'(1)\sum_{2^{\ell-1}\leq n<2^{\ell}} \frac1{n^{m+1}}\\
          &-c_m(1)\sum_{2^{\ell-1}\leq n<2^{\ell}} \frac{\log n}{n^{m+1}}
        \end{aligned}\right.
  \end{equation}
  One checks easily from Equation \eqref{eq:cm} that $c_m(1) = (m+1)^{-1}$ for
  $m\geq0$.  Leaving temporarily aside the
  consideration of the derivative $c_m'(1)$, we compute:
  \begin{align*}
   &\sum_{m=0}^\infty(-1)^m\biggl(
          \frac{-1}{m+1}\sum_{2^{\ell-1}\leq n<2^{\ell}} \frac{\log n}{n^{m+1}}
          + \frac{H_m}{m+1}\sum_{2^{\ell-1}\leq n<2^{\ell}} \frac1{n^{m+1}}
       \biggr)
\\
  &=\sum_{2^{\ell-1}\leq n<2^{\ell}}
   \left(-\log n \log(1 + \frac1n) 
   + \left.\dds\right|_{s=1} 
        \sum_{m=1}^\infty\frac{(-1)^m}{m+1}\frac{(s)_m}{m!}(1/n)^{m+1} \right)
\\
  &=\sum_{2^{\ell-1}\leq n<2^{\ell}}
   \left(-\log n \log(1 + \frac1n) 
    + \left.\dds\right|_{s=1} \int_0^{\frac1n}(1 + x)^{-s}\dx\right)
\\
  &=\sum_{2^{\ell-1}\leq n<2^{\ell}}
   \left(-\log n \log(1 + \frac1n) + 
   \int_0^{\frac1n}\frac{-\log(1+x)}{(1 + x)}\dx\right)
\\
&=
\sum_{2^{\ell-1}\leq n<2^{\ell}}
   \left(-\log n \log(1 + \frac1n) - \frac12\log^2(1+\frac1n)\right)
\\
&=
-\frac12\sum_{2^{\ell-1}\leq n<2^{\ell}} \log(1+\frac1n)\log \bigl(n^2(1+\frac1n)\bigr)
=
-\frac12\sum_{2^{\ell-1}\leq n<2^{\ell}} \bigl(\log^2(n+1)-\log^2(n)\bigr)
\\
&=
-\frac12\Bigl(\ell^2\log^22 - (\ell-1)^2\log^22\Bigr) = -(\ell-\frac12)\log^22.
  \end{align*}
Hence, from Equation~\eqref{eq:der1}:
\begin{equation*}
  \eta'(1) =  (\log 2) H_{2^{\ell-1}-1}  -(\ell-\frac12)\log^22
 + \sum_{m=0}^\infty(-1)^m c_m'(1)\sum_{2^{\ell-1}\leq n<2^{\ell}} \frac1{n^{m+1}}.
\end{equation*}
Now, as is well-known, $\eta'(1) = \log(2)(\gamma - \frac12 \log2)$, so, with
the definition
\begin{equation}\label{eq:emdef}
  e_m = -(m+1)\frac{c_m'(1)}{\log 2},
\end{equation}
(in particular $e_0=0$), we obtain
\begin{equation*}
  \gamma =  H_{2^{\ell-1}-1}- (\ell - 1)\log 2 
  + \sum_{m=1}^{\infty}(-1)^{m-1}\frac{e_m}{m+1}\sum_{2^{\ell-1}\leq n<2^{\ell}} \frac1{n^{m+1}}.
\end{equation*}
The proof of Theorem~\ref{thm:main} will be complete once it has been
confirmed that the $(e_m)$ sequence defined by Equation~\eqref{eq:emdef} verifies the
recurrence~\eqref{eq:emrec}. Taking the derivative at $s=1$ of Equation
\eqref{eq:cm}, and recalling $c_m(1) = (m+1)^{-1}$, we obtain, for $m\geq1$:
\begin{align*}
  c_m'(1) &= \frac{-(\log 2)2^{m+1}}{2^{m+1}-2}c_m(1) 
            +(2^{m+1}-2)^{-1}\sum_{j=1}^m\binom mj c_{m-j}'(1)\\
e_m &= \frac{2^{m+1}}{2^{m+1}-2} +
      (2^{m+1}-2)^{-1}\sum_{j=1}^m\binom mj\frac{m+1}{m-j+1}e_{m-j}.
\end{align*}
This, indeed, is Equation~\eqref{eq:emrec}.  And $e_0=0$.
\end{proof}

\section{Elementary bounds}

In this section, we mention two elementary estimates which can be established
using only the recurrence~\eqref{eq:emrec}. As such results are superseded for
large $m$ (if one goes through the proof to get explicit bounds) by Theorem
\ref{thm:em} and by explicit numerical results for $m$ of moderate size, we
omit the proofs and refer the interested reader to
\url{https://arxiv.org/abs/2603.29998v1}.  We note that the main term
$(\log 2)^{-1}H_{m+1}$ is indeed a better reference point than $\log_2(m)$ in
view of Theorem \ref{thm:fullasymp} (and the computations leading to it).
\begin{prop}\label{prop:1}
  There holds for any $m\in\NN$ (with $H_n=\sum_{j=1}^n j^{-1}$):
  \begin{equation}\label{eq:propa}
    \frac{H_{m+1} - 1}{\log 2} \leq e_m
    < \frac{H_{m+1}}{\log 2} -\np{0.161}\,.
  \end{equation}
  More precisely, for $m\geq2$:
  \begin{equation}\label{eq:propb}
    \frac{H_{m+1}}{\log 2} - \np{0.35} < e_m <  \frac{H_{m+1}}{\log 2} - \np{0.31}\;.
  \end{equation}
\end{prop}

\section{The coefficients \texorpdfstring{$e_m$}{em} via power sums}

We have seen in the proof of Theorem~\ref{thm:main} that we can define $e_m$
as $-(m+1)c_m'(1)/\log 2$ where the holomoprhic functions $c_m(s)$ on the
half-plane $\Re s>0$ verify Equation~\eqref{eq:cm}.  Further, we explained
that $c_m(s) = \frac{2^s-2}{2^s}u_m(s)$ (in particular $u_0(s) =
\sum_{l=0}^\infty 2^{-l(s-1)} = 2^s/(2^s-2)$), with some sequence
$(u_m(s))_{m\geq0}$ of meromorphic functions which is defined in
\cite{burnolzeta_umosc}.  It is mentioned in \cite[\S2,
Eq. (3)]{burnolzeta_umosc}, following up on \cite[\S2, Def.\@3]{burnolzeta},
that for $\Re s>1$, $u_m(s)$ is the $m$th moment of a certain discrete
(complex) measure with support in $\Ifo 01$:
\begin{equation}
  \label{eq:um}
  \Re s>1\implies u_m(s) =
   \sum_{l=0}^\infty \biggl(\sum_{0\leq k <2^l}\bigl(\frac k{2^l}\bigr)^m\biggr) 2^{-ls}\,.
\end{equation}
As the function $t\mapsto t^m$ is non-decreasing and convex, there
holds
\begin{equation}\label{eq:rectangles}
  0\leq 
  \int_0^1 t^m\dt - 2^{-l}\sum_{0\leq k <2^l}\bigl(\frac k{2^l}\bigr)^m 
  \leq \frac12 2^{-l}.
\end{equation}
So, at first for $\Re s > 1$:
\begin{align*}
  u_m(s) &= (m+1)^{-1} \sum_{l=0}^\infty 2^{-l(s-1)} - 
\sum_{l=0}^\infty \biggl((m+1)^{-1} - 2^{-l}\sum_{0\leq k <2^l}\bigl(\frac k{2^l}\bigr)^m\biggr)
2^{-l(s-1)}
\\
(m+1)u_m(s) &=
\frac{2^s}{2^s-2} - 
\sum_{l=0}^\infty \biggl(1 - (m+1)2^{-l}\sum_{0\leq k <2^l}\bigl(\frac k{2^l}\bigr)^m\biggr)
2^{-l(s-1)}
\\
(m+1)c_m(s) &= 1 -\frac{2^s-2}{2^s}
\sum_{l=0}^\infty \biggl(1 - (m+1)2^{-l}\sum_{0\leq k <2^l}\bigl(\frac k{2^l}\bigr)^m\biggr)
2^{-l(s-1)}.
\end{align*}
Thanks to estimate~\eqref{eq:rectangles}, this gives the analytic continuation
to $\Re s>0$, and we can now compute the value of the derivative at $s=1$:
\begin{equation*}
  (m+1)c_m'(1) = -\log(2)\sum_{l=0}^\infty \biggl(1 - 
   (m+1)2^{-l}\sum_{0\leq k <2^l}\bigl(\frac k{2^l}\bigr)^m\biggr).
\end{equation*}
From Equation~\eqref{eq:emdef} which arose in the course of the proof of
Theorem~\ref{thm:main}, we thus get an ``explicit'' formula for the coefficients $e_m$:
\begin{prop}
  There holds:
  \begin{equation}
\label{eq:emformula}
    e_m = \sum_{l=0}^\infty \biggl(1 - (m+1)2^{-l}\sum_{0\leq k <2^l}\bigl(\frac k{2^l}\bigr)^m\biggr)\,.
\end{equation}
\end{prop}
In terms of Bernoulli polynomials and numbers:
\begin{equation}\label{eq:ember2}
e_m = \sum_{l=0}^\infty \biggl(1 - \frac{B_{m+1}(2^l)-B_{m+1}(0)}{2^{(m+1)l}}\biggr)
=
-\sum_{k=1}^m \binom{m+1}{k}B_k\frac{2^k}{2^k-1}\,.
  \end{equation}
  To obtain Equation~\eqref{eq:ember2},
  we use the properties of Bernoulli numbers and
  polynomials as summarized below, and  $2^k/(2^k-1)$ is obtained as 
  $\sum_{l=0}^\infty 2^{-kl}$. As $B_{m+1}(1)-B_{m+1}(0)=0$ if $m>0$,
  the $l=0$ term contributes $+1$ to $e_m$,
  and there also
  holds:
\begin{equation*}
  m\geq1\implies e_m = 1 - \sum_{k=1}^m \binom{m+1}{k}B_k\frac{1}{2^k-1}\,.
\end{equation*}
  The Bernoulli polynomials and numbers used in~\eqref{eq:ember2} are defined
  by
  \begin{equation*}
\frac{t\e^{xt}}{\e^t-1}= \sum_{n=0}^\infty B_n(x)\frac{t^n}{n!}.
  \end{equation*}
They verify $B_{n}(x+1)-B_{n}(x)=nx^{n-1}$, hence
\begin{equation*}
\sum_{0\leq k<n}k^m=\frac{B_{m+1}(n)-B_{m+1}(0)}{m+1}\,.
\end{equation*}
Also 
  \begin{equation*}
    B_{m+1}(x)=\sum_{k=0}^{m+1}\binom{m+1}{k}B_kx^{m+1-k}\,,
  \end{equation*}
and  $B_0(x) = 1$, $B_1(x) = x -\frac12$, $B_n = B_n(0)$,
  $B_{2p}=(-1)^{p-1}|B_{2p}|$ ($p\geq1$),
    $B_{2p+1}=0$ ($p\geq1$). That $B_n$ is used both for the polynomial and
    its value at zero should not cause confusion.
\begin{rema}
A variant of~\eqref{eq:ember2} is:
  \begin{equation}
    e_m = \sum_{l=0}^\infty (l+1)\frac{B_{m+1}(2^{l+1})-2^{m+1}B_{m+1}(2^l) 
                           + (2^{m+1}-1)B_{m+1}(0)}
                          {2^{(l+1)(m+1)}}
  \end{equation}
  One only needs to write the numerator as $(l+1)(v_{l+1}-2^{m+1}v_l)$ with
  $v_l=B_{m+1}(2^l)-B_{m+1}(0)-2^{(m+1)l}$ and rearrange the sum.  Details are
  left to the reader.  In terms of the Euler polynomials (see
  Equation~\eqref{eq:eulerber}), this is:
  \begin{equation*}
       \frac{m+1}2\sum_{l=0}^\infty (l+1)\frac{E_m(2^{l+1})-E_m(0)}
                          {2^{(l+1)(m+1)}}\,.
  \end{equation*}
\end{rema}

\section{Approximations to order \texorpdfstring{$O(1/m)$}{O(1/m)}}

Let us make the general definition, for $n$ a positive integer:
\begin{equation}\label{eq:pm}
  p_m(n) =  n^{-1}\sum_{0\leq k <n} (n^{-1}k)^m\,.
\end{equation}
They are the moments of the discrete probability measure $\sum_{0\leq k <
  n}n^{-1}\delta_{k/n}$.  It is not needed here for $m$ to be an integer, and
in what follows we let $e_m$ for $m$ real positive be given by Equation
\eqref{eq:emformula}, i.e.:
\begin{equation}\label{eq:empm}
  e_m = \sum_{l=0}^\infty \bigl(1 - (m+1)p_m(2^l)\bigr)\,,
\end{equation}
and we study the asymptotic for real positive $m$ going to
infinity.

There holds
\begin{equation}
  \label{eq:pmber}
  (m+1)p_m(n) = \frac{B_{m+1}(n)-B_{m+1}(0)}{n^{m+1}}\,.
\end{equation}

We  will use the same inequalities as in \cite[Prop.\@ 1]{burnolosc}:
\begin{align}
\notag
    0\leq x &\implies 0\leq \e^{-m(1-x)}- x^m\,,
\\
\label{eq:foo}
    \frac12\leq x\leq 1&\implies 0\leq \e^{-m(1-x)}- x^m\leq m(1-x)^2\e^{-m(1-x)}\,.
\end{align}
Inequality~\eqref{eq:foo}, with $x=1 - n^{-1}k$, $1\leq k \leq \frac n2$, is
found in a similar context also in \cite[answer to 6.3-34]{knuth}.  It
suggests approximating $p_m(n)$ by
\begin{equation*}
  n^{-1}\sum_{0\leq k <n} \e^{-m(1 - n^{-1}k)} = \frac{1 - \e^{-m}}{n(\e^{n^{-1}m}-1)}\,,
\end{equation*}
which we rewrite as $(1 - \e^{-m})q_m(n)$, defining for this purpose:
\begin{equation}\label{eq:qm}
  q_m(n) = \frac{1}{n(\e^{n^{-1}m}-1)}\,.
\end{equation}
But the approximation of $p_m(n)$ by $q_m(n)$ becomes poor
for $n$ large, as:
\begin{equation*}
  \lim_{n\to\infty} p_m(n) = \int_0^1 x^m\dx = (m+1)^{-1}, \text{ while }
  \lim_{n\to\infty} q_m(n) = m^{-1}.
\end{equation*}
Thus, we should either switch to $q_{m+1}(n)$ (we shall see later that it
arises very naturally from Equation~\eqref{eq:pmber}), or approximate
$(m+1)p_m(n)$ via $mq_m(n)$.

We start with a Lemma which uses Equation~\eqref{eq:foo} to compare $p_m(n)$
with $(1-\e^{-m})q_m(n)$ uniformly for the entire $n$ range.
\begin{lem}\label{lem:bound1}
There is a positive constant
$K_1<3.1$ such that for every real positive $m$ and every positive integer
$n$, there holds:
\begin{equation}\label{eq:K1}
  0 \leq (1-\e^{-m})q_m(n) - p_m(n) \leq \frac{K_1}{m^2}\,.
\end{equation}
\end{lem}
\begin{proof}
We write:
\begin{equation*}
  (1-\e^{-m})q_m(n) - p_m(n)
  \begin{aligned}[t]
    &= \Bigl(n^{-1}\sum_{0\leq k < n/2} + n^{-1}\sum_{n/2\leq k <n}\Bigr)
    \bigl(\e^{-m(1-n^{-1}k)} - (n^{-1}k)^m\bigr)\\
    &= S_1 + S_2\,.
  \end{aligned}
\end{equation*}
We bound $S_1$ (trivially) for $n$ even by $\int_0^{\frac12}\e^{-m(1-x)}\dx
=m^{-1}(\e^{-m/2}- \e^{-m})$ and for $n$ odd by
$\int_0^{\frac12-(2n)^{-1}}\e^{-m(1-x)}\dx +
2\int_{\frac12-(2n)^{-1}}^{\frac12}\e^{-m(1-x)}\dx =
m^{-1}\bigl(\e^{-m/2}-\e^{-m} + \e^{-m/2}(1 - \e^{-m/(2n)})\bigr)$.  This is
maximal for $n=1$ so we adopt $2m^{-1}(\e^{-m/2}-\e^{-m})$ as general upper
bound valid for all positive integers $n$ and positive real numbers
$m$. Regarding $S_2$, we use the inequality~\eqref{eq:foo}, and after the
change $k\to n-k$, and then extending the range of $k$ to infinity and letting
$\tau=n^{-1}m>0$, we get
\begin{equation*}
  S_2 \leq \sum_{1\leq k\leq n/2} mn^{-3}k^2\e^{-n^{-1}mk} 
     \leq m^{-2}\sum_{k=1}^\infty\tau^3 k^2\e^{-\tau k}
    =m^{-2}\tau^3\frac{\e^{-\tau}+\e^{-2\tau}}{(1-\e^{-\tau})^3}\,.
\end{equation*}
It turns out that this function of $\tau>0$ is nicely decreasing with its
supremum attained for $\tau\to0^+$ (this can be proven using the Poisson
summation formula).  So $m^2S_2<2$.
Adding the bound for $S_1$ we obtain 
\eqref{eq:K1} with
\begin{equation}
  K = 2 + \sup_{x>0}2x(\e^{-x/2}-\e^{-x}).
\end{equation}
% ...     
The supremum is attained at $x_0\approx\np{2.89115}$ (and values decrease for
$x>x_0$) and is
$\approx\np{1.04138}$.  So we can take
$K=\np{3.05}$. Restricting to $m\geq m_0\geq 3$, we get~\eqref{eq:K1} with
$K=2+2m_0(\e^{-m_0/2}-\e^{-m_0})$, and only $2+m_0(\e^{-m_0/2}-\e^{-m_0})$ for
even $n$'s.
\end{proof}
As a corollary, if we let both $m$ and $n$ go to infinity such that their ratio
converges to some value $q$, we get (with $1$ on the right if $q=0$):
\begin{equation*}
  \begin{split}
    \lim_{m,n\to\infty, mn^{-1}\to q} \frac{B_{m+1}(n)-B_{m+1}(0)}{n^{m+1}} =
    \lim \,(m+1)p_m(n) \\= \lim\, mq_m(n)
= \lim \frac{n^{-1}m}{\e^{n^{-1}m}-1} =
    \frac q{\e^q-1}\,.
  \end{split}
\end{equation*}
This is indeed what one would
expect from expressing the Bernoulli polynomials in terms of Bernoulli numbers
and taking the limit term-wise, but Lemma~\ref{lem:bound1} makes it easy (and
gives an $O(m^{-1})$ error estimate).  Reference \cite{lopeztemme1999}, which
uses the saddle-point method in the complex domain, provides full asymptotic
for, among others, $B_n(nz+\frac12)$ as $n\to\infty$.

\begin{lem}\label{lem:bound2}
There is a constant $K_2$ such that for every positive $m$:
\begin{equation*}
0\leq  \sum_{l=0}^{\lfloor \log_2(m+1)\rfloor}
  \Bigl((1-\e^{-m})q_m(2^l) - p_m(2^l)\Bigr) \leq \frac{K_2}{(m+1)^2}\,.
\end{equation*}
\end{lem}
\begin{proof}
  Let $L_m = \lfloor \log_2(m+1)\rfloor$.  The $l=0$ term on the left-hand
  side of Equation~\eqref{eq:K1} is, for $m>0$, $(1-\e^{-m})q_m(1)=(1-\e^{-m})
  (e^m-1)^{-1}=\e^{-m}\leq 4\e^{-1}(m+1)^{-2}$.

  We assume now $m\geq1$, so that $L_m\geq 1$, and we need only consider
  $1\leq l \leq L_m$.  Using the same notation as in the proof of
  Lemma~\ref{lem:bound1}, and Equation~\eqref{eq:foo}, we have, for any even
  positive integer $n$:
\begin{equation*}
  0\leq (1-\e^{-m})q_m(n) - p_m(n)\leq
    m^{-1}\e^{-m/2} 
  + m^{-2}\sum_{k=1}^\infty\tau^3 k^2\e^{-\tau k}\,,
\end{equation*}
with $\tau = n^{-1}m$.

Summing for $2\leq 2^l\leq m+1$, the first term in the right-hand side
contributes $L_m m^{-1}\e^{-m/2}$ which is a fortiori $O((m+1)^{-2})$.  As per
the second summands, as $\tau\geq m/(m+1)\geq \frac12$, there holds
$\e^{-\tau}\leq \e^{-\frac12}$ and $\sum_{k\geq1}k^2\e^{-\tau k}\leq
\e^{-\tau}\sum_{k\geq1}k^2\e^{-(k-1)/2}$.  Hence, the second summand is, for
each of $\tau=\frac {m}{2}$,  $\frac{m}{4}$, \dots, $2^{-L_m}\geq \frac12$, bounded by $C
m^{-2}\tau^3\e^{-\tau}$, for some constant $C$.  We use the weaker upper bound
$m^{-2}O(\e^{-\tau/2})$ and obtain in total
$m^{-2}O(\e^{-1/4}+\e^{-1/2} + \dots + \e^{-2^{L_m}/8}) = O(m^{-2})$.
\end{proof}
Let us set from here on,  as in the previous proof:
\begin{equation*}
L_m = \lfloor \log_2(m+1)\rfloor\,.
\end{equation*}
We also need an upper bound for
$
  m\sum_{l=0}^{L_m}q_m(2^l) = \sum_{l=0}^{L_m} \frac{m}{2^l (\e^{2^{-l}m}-1)}
$.
Not aiming at optimal results, we use the trivial bound valid for positive
$x$:
\begin{equation*}
  \frac{x}{\e^x-1} = \frac{x\e^{-x/2}}{2\sinh \frac{x}{2}}<\e^{-x/2}.
\end{equation*}
Hence, as in the proof of Lemma~\ref{lem:bound2}, we get, supposing $m\geq1$,
hence $2^{L_m}\leq 2m$:
\begin{align*}
  m \sum_{l=0}^{L_m}q_m(2^l)&\leq 
\e^{-\frac m2} + \e^{-\frac m4} + \dots + \e^{-2^{-L_m}\frac m2}
\\
&\leq 
\e^{-2^{L_m}/4} + \e^{-2^{L_m}/8} + \dots + \e^{-1/4} = O(1).
\end{align*}
For $0<m<1$, there is only one term $\frac{m}{\e^m-1} = O(1)$.  Hence:
\begin{lem}\label{lem:bound3}
  There is a constant $K_3$ such that for every positive $m$:
\begin{equation*}
  0\leq \sum_{l=0}^{L_m} q_m(2^l) \leq \frac{K_3}{m}\,.
\end{equation*}
\end{lem}
\begin{lem}\label{lem:bound4}
There is a constant $K_4$ such that for every positive $m$:
\begin{equation*}
  0\leq \sum_{l=0}^{\lfloor \log_2(m+1)\rfloor}
  \Bigl(q_m(2^l) - p_m(2^l)\Bigr) \leq \frac{K_4}{m^2}\,.
\end{equation*}
\end{lem}
\begin{proof}
From Lemmas~\ref{lem:bound2} and~\ref{lem:bound3} the sum is bounded above by 
$K_2(m+1)^{-2}+\e^{-m}K_3 m^{-1}$, so, multiplying by $m^2$ it is bounded
by $K_2 + \e^{-1}K_3$.
\end{proof}

We now obtain:
\begin{lem}\label{lem:bound5}
  There is a constant $K_5$ such that for every positive $m$:
\begin{equation*}
  \sum_{l=0}^{L_m} \Bigr|mq_m(2^l) - (m+1)p_m(2^l)\Bigl|
\leq \frac{K_5}{m+1}
\end{equation*}
\end{lem}
\begin{proof}
For $0<m<1$, we have only one term $mq_m(1)-(m+1)p_m(1)$ with
$mq_m(1) = m(e^m -1)^{-1}\leq 1$, and $p_m(1) = 0$, so it is bounded by $1$
hence by $K(m+1)^{-1}$ for any  $K\geq2$.
For $m\geq1$, we first decompose
\begin{align*}
  mq_m(2^l) - (m+1)p_m(2^l) =
\begin{aligned}[t]
  &(m+1)\Bigl((1-e^{-m})q_m(2^l) - p_m(2^l)\Bigr) \\
  &-\bigl(1 - (m+1)\e^{-m}\bigr) q_m(2^l)
\end{aligned},
\end{align*}
hence, 
\begin{equation*}
  \Bigl|mq_m(2^l) - (m+1)p_m(2^l)\Bigr|
\leq (m+1)\Bigl((1-e^{-m})q_m(2^l) - p_m(2^l)\Bigr) + q_m(2^l)\,.
\end{equation*}
Using Lemmas~\ref{lem:bound2} and~\ref{lem:bound3} we get, recalling $m\geq1$,
\begin{equation*}
  \sum_{l=0}^{L_m} \Bigr|mq_m(2^l) - (m+1)p_m(2^l)\Bigl|
\leq \frac{K_2}{m+1} + \frac{K_3}{m} \leq \frac{K_2 + 2 K_3}{m+1}\,,
\end{equation*}
and the statement to be proven holds with $K_5 = \max(2, K_2 + 2K_3)$.
\end{proof}
We now need to control the $2^l >m+1$ range.  This is done via a completely
different approach to the quantities $p_m(n)$ and $q_m(n)$ (or $q_{m+1}(n)$).
\begin{lem}\label{lem:bound6}
  There are constants $K_6$ and $K_7$ such that
  \begin{align}\label{eq:K6}
  \sum_{l>L_m}\Bigr|mq_m(2^l) - (m+1)p_m(2^l)\Bigl|
&\leq \frac{K_6}{m+1}
\\\label{eq:K7}
  \sum_{l>L_m}\Bigr|(m+1)q_{m+1}(2^l) - (m+1)p_m(2^l)\Bigl|
&\leq \frac{K_7}{m+1}
  \end{align}
\end{lem}
\begin{rema}
  In a certain way, Equation~\eqref{eq:K7} is the one arising
  more naturally from the proof.  But as we have already written up
  a proof for  Lemma~\ref{lem:bound5}, which is expressed with the
  $q_m(2^l)$'s, it is convenient to also prove Equation~\eqref{eq:K6}.
\end{rema}
\begin{proof}
  Let $n$ be some positive integer. From Equation~\eqref{eq:pmber}:
  \begin{equation*}
    (m+1)p_m(n) = \frac{B_{m+1}(n)-B_{m+1}(0)}{n^{m+1}} = \sum_{k=0}^m
    \binom{m+1}k B_kn^{-k}\,.
  \end{equation*}
We now approximate the binomial coefficients:
  \begin{align*}
    \binom{m+1}k &= \frac{1}{k!}(m+1)^k\prod_{0\leq p < k}(1 - (m+1)^{-1}p)
\\
\implies 0&\leq \frac{(m+1)^{k}}{k!}-\binom{m+1}k\leq \frac{k(k-1)}{2(m+1)}\frac{(m+1)^{k}}{k!}\,.
  \end{align*}
We used $\prod_{0\leq p<k}(1-\epsilon_p)\geq 1 - \sum_{0\leq p\leq
  k}\epsilon_p$ which is valid for $0\leq \epsilon_p\leq 1$, $0\leq p
<k$. Hence
\begin{equation*}
  \Bigl| \sum_{k=0}^m \frac{B_k}{k!}(m+1)^kn^{-k} - (m+1)p_m(n)\Bigr|
\leq (m+1)^{-1}\sum_{k=0}^m \frac{|B_k|}{k!}\frac{k(k-1)}2 \bigl((m+1)n^{-1}\bigr)^k
\end{equation*}
We assume from here on $n\geq m+1$, which is the case for the powers of $2$ we
will consider, $n=2^l$, $l>L_m$.  The finite sums in the previous inequality
are then partial sums of convergent series. We replace $k(k-1)$ by the larger
$k(k+1)$ which verifies $m+1\leq \frac{k(k+1)}{2}$ for $k\geq m+1$, $m>0$. This
allows to extend our termwise estimates to the $k>m$ range. Thus,
there holds:
\begin{equation*}
  \Bigl| \sum_{k=0}^\infty \frac{B_k}{k!}(m+1)^kn^{-k} - (m+1)p_m(n)\Bigr|
\leq (m+1)^{-1}\sum_{k=2}^\infty \frac{|B_k|}{k!}\frac{k(k+1)}2 \bigl((m+1)n^{-1}\bigr)^k.
\end{equation*}
We recognize in the left-hand side
\begin{equation*}
   \sum_{k=0}^\infty \frac{B_k}{k!}(m+1)^kn^{-k} =
   \frac{n^{-1}(m+1)}{\e^{n^{-1}(m+1)}-1} = (m+1)q_{m+1}(n)\,.
\end{equation*}
We can bound $S(\tau)=\sum_{k\geq2}^\infty \frac{|B_k|}{k!}\frac{k(k+1)}2
\tau^k$ for $|\tau| \leq 1$ by $C_1|\tau|^2$ with $C_1 = \sum_{k\geq2}^\infty
\frac{|B_k|}{k!}\frac{k(k+1)}2<\infty$. Hence:
\begin{equation*}
  n\geq m+1\implies \Bigl| (m+1)q_{m+1}(n) - (m+1)p_m(n)\Bigr|
\leq C_1 (m+1)^{-1}\bigl((m+1)n^{-1}\bigr)^2.
\end{equation*}
For $n=2^l$, $l>L_m$, we have $n>(m+1)$ indeed.  And summing
$\bigl((m+1)n^{-1}\bigr)^2$ over such powers of $2$ is bounded above by $1 +
4^{-1} + 16^{-1} + \dots$.  Hence Equation~\eqref{eq:K7}.

A
modification in the handling the binomial coefficients will
readily give Equation~\eqref{eq:K6}. We write this time:
  \begin{equation*}
    \binom{m+1}k = \frac{1}{k!}m^k\prod_{-1\leq p < k-1}(1 - m^{-1}p)\,.
\end{equation*}
The product (depending on $k\in\Iff{0}{m}\cap\ZZ$) is bounded above by $1+m^{-1}$
and below by $\prod_{0\leq p < k-1}(1 - m^{-1}p)$, hence the distance to $1$
is bounded in absolute value by
\begin{equation*}
  m^{-1}\delta_k, \qquad \text{with }\delta_k = \max(1, \frac{(k-1)(k-2)}{2})\qquad(k\geq1).
\end{equation*}
Moreover, $\delta_0=0$.  We can use a weaker upper bound $k(k+1)/2$, so that
$k(k+1)/2 \geq m$ holds for $k\geq m+1$.  This avoids having to discuss
separately the (exponentially small, as one can check) contribution from the
$k>m$ range.  The sole differences with our earlier situation is that $k$ now
starts at $1$ not at $2$, and that powers of $m+1$ are replaced by powers of
$m$, so that we have to consider the sum of $mn^{-1}$, not of
$\bigl((m+1)n^{-1}\bigr)^2$, over $n=2^l>m+1$.  We obtain then Equation~\eqref{eq:K6}.
\end{proof}

\begin{prop}\label{prop:emfm}
Let, for positive real $m$,
\begin{equation}
    \label{eq:fm}
  f_m = \sum_{l=0}^\infty \bigl(1 - \frac{m2^{-l}}{\e^{m2^{-l}}-1}\bigr)\,.
\end{equation}
 There holds $|e_m - f_m|\leq \frac{K}{m+1}$ for all $m>0$ and some constant
 $K$.
\end{prop}
\begin{proof}
  The series defining $f_m$ has positive terms and is convergent.
  Lemmas~\ref{lem:bound5} and~\ref{lem:bound6} allow to bound
  $|e_m-f_m|=\Bigl|\sum_{l=0}^\infty \bigl(mq_m(2^l)-(m+1)p_m(2^l)\bigr)\Bigr|$,
  even after having moved the absolute values inside the summation.
\end{proof}


  
\section{Interlude on another sequence}\label{sec:knuth}

As part of \cite[Ex.\@ 6.3-34]{knuth} the following is asked (we have
  translated to our notation%
\footnote{Mind that this $a_m$ is not exactly the $a_n$ from Equation (18) of
\cite[\S6.3]{knuth}.}).
\emph{Let for $m\geq2$
  \begin{align*}
    a_m &= \sum_{l=1}^\infty\biggl(\sum_{1\leq k<2^l}(2^{-l}k)^{m-1}
    -
    \frac{2^{l}}{m} + \frac12\biggr)
\\
    b_m &= \sum_{l=1}^\infty \biggl(\frac{1}{\e^{2^{-l}m}-1}
    -\frac{2^{l}}{m} + \frac12\biggr).
  \end{align*}
  Show $a_m = b_m + O_{m\to\infty}(m^{-1})$.} We note that the general term
$a_m(l)$ in
the $a_m$ series is $O(2^{-l})$. This follows from Euler-Maclaurin, or
directly from using Bernoulli polynomials:
\begin{equation}\label{eq:aml}
  a_m(l) = \frac{B_m(2^l)-B_m(0)}{m 2^{l(m-1)}} -
    \frac{2^{l}}{m} + \frac12 = \sum_{k=2}^{m-1}\binom{m}{k}m^{-k} B_k m^{k-1}(2^{-l})^{k-1}\,.
\end{equation}
We used $B_0 = 1$, $B_1 = -\frac12$ and isolated a factor $m^{-k}$ to merge it
with the binomial coefficient.
Similarly to the proof of Lemma~\ref{lem:bound6}, we now take note of the inequality
\begin{equation*}
  0\leq\frac {1}{k!} - \frac{\binom{m}{k}}{m^k}\leq \frac 1{k!}\frac{k(k-1)}{2m}\,.
\end{equation*}
Let $b_m(l)$ be the $l$th term in $b_m$.
We see that we are going to obtain $\sum_{2^l >m} |b_m(l)-a_m(l)| =
O(m^{-1})$ via the similar steps as in the proof of Lemma~\ref{lem:bound6}.
It is important that in Equation~\eqref{eq:aml}, the smallest exponent of
$m2^{-l}$ is $1>0$.

There remains to evaluate the contributions for
$1\leq 2^l \leq m$.  Using the notation from the previous section, it is a question
of upper bounding
\begin{equation*}
  \sum_{1\leq 2^l\leq m}\left| 2^l p_{m-1}(2^l) - 2^l q_m(2^l)\right|.
\end{equation*}
This is at most:
\begin{equation*}
  S_1 + S_2 = \sum_{1\leq 2^l\leq m}\left| 2^l p_{m-1}(2^l) - 2^l q_{m-1}(2^l)\right| +
  \sum_{1\leq 2^l\leq m} \left|2^l q_{m-1}(2^l) - 2^l q_{m}(2^l)\right|\,.
\end{equation*}
For $S_1$, we use $2^l\leq m$ and Lemma~\ref{lem:bound4} to see
that $S_1\leq mK_4(m-1)^{-2} = O(1/m)$.
Regarding $S_2$, we estimate, for each positive integer $n$:
\begin{align*} 
(\e^{n^{-1}(m-1)}-1)^{-1}-(\e^{n^{-1}m}-1)^{-1}
&=
\frac{\e^{n^{-1}\xi}}{n(\e^{n^{-1}\xi}-1)^2}\quad(\xi\in\Iff{m-1}m)
\\
&= \frac1{n 4 \sinh^2((2n)^{-1}\xi)}
\leq \frac{n}{\xi^2}\leq \frac{n}{(m-1)^2}\,.
\end{align*}
So $S_2\leq 2m/(m-1)^2 = O(1/m)$ and the answer to \cite[Ex.\@ 6.3-34]{knuth}
is complete.%
\footnote{See also the errata list, updated in May 2026,
  at \url{https://www-cs-faculty.stanford.edu/~knuth/taocp.html}, about the
  answer on page 727 of \cite{knuth} for this Exercise 6.3-34.}%
\textsuperscript{,}%
\footnote{Theorem 4.1 of \cite{szp1987} looked at first sight as being a
  general estimate having Exercice 6.3-34(b) as a special case.  But the
  $T(n,r,s)$ of that reference approximates in general the $S(n,r,s)$ only to
  order $O(n^{-s})$, not $O(n^{-s-1})$ as claimed.  So, using $d = 2$, $q =
  1$, $r = 0$, $s = 0$, $t = 2$ in the notations of that reference, this
  Theorem only provides $O(1)$ for Exercice 6.3-34.  Similarly with $d = 2$,
  $q = 1$, $r = 0$, $s = 1$, $t = 1$, then $T(n,r,s)-S(n,r,s)$ is related with
  $(e_n - f_n)/(n+1)$ in our notation and provides only $e_n-f_n=O(1)$.
  Details available from the author.}



\section{The oscillating term}\label{sec:gm}

Clearly the series from Equation~\eqref{eq:fm} defining $f_m$ converges
uniformly for $m$ bounded (for now, $m$ is real positive).  We define (at first, for
positive real $m$):
\begin{equation}
\label{eq:gm}
  g_m  = -  \sum_{l=-1}^{-\infty} \frac{m2^{-l}}{\e^{m2^{-l}}-1}\,.
\end{equation}
We
bound trivially the terms in the series defining $-g_m$:
\begin{gather}\notag
  x>0\implies \frac{x}{\e^x -1}= \frac{x \e^{-x/2}}{2 \sinh(x/2)}\leq \e^{-x/2}
\\
\label{eq:gmbound}
 0\leq - g_m \leq \frac{2m}{\e^{2m}-1} + \sum_{n=4}^\infty \e^{-nm/2} = O_{m\to\infty}(me^{-2m}).
\end{gather}
And this also showed that the series defining $g_m$ is uniformly convergent
for $m\geq a>0$.  We observe that $f_{2m} = f_m + (1 - \frac{2m}{\e^{2m}-1})$
and $g_{2m} = g_m + \frac{2m}{\e^{2m}-1}$, so $f_{2m}+g_{2m} = 1 + f_m + g_m$.
This motivates the definition of the $1$-periodic function on $\Ioo0\infty$:
\begin{equation}\label{eq:phi}
  \phi(t) = -t -\frac12 + \sum_{l=0}^\infty \bigl(1 - \frac{2^{t-l}}{\e^{2^{t-l}}-1}\bigr)
  - \sum_{l=-1}^{-\infty} \frac{2^{t-l}}{\e^{2^{t-l}}-1}\,.
\end{equation}
The reason for the extra term $-\frac12$ is that this makes the average of
$\phi$ vanish, as will be seen later.
As shown previously the two series are uniformly convergent for $0\leq t\leq
1$, so $\phi$ is a continuous function. Here is a stronger (elementary) statement:
\begin{prop}\label{prop:phiana}
  Equation~\eqref{eq:phi} defines $\phi(t)$ as an analytic function of
  $t$ in the horizontal strip $|\Im t| < \frac{\pi}{2\log(2)}$.
\end{prop}
\begin{proof}
  Details are left to reader: $\sum_{l=0}^\infty (1 -
  \frac{2^{-l}z}{\e^{2^{-l}z}-1})$ defines a meromorphic function in the
  entire complex plane, with poles at the elements of $2\pi i
  \ZZ\setminus\{0\}$, so in particular if we substitute $z=2^t$, we obtain an
  analytic function in the strip $|\Im t|<\frac{\pi}{2\log 2}$.  Regarding
  $\sum_{k=1}^\infty \frac{2^k2^t}{\e^{2^k2^t}-1}$, it is absolutely and
  uniformly convergent if $a<\Re t<b$, $|\Im
  t|<\frac\pi{2\log(2)}(1-\epsilon)$, $0<\epsilon<1$ (so that $\Re
  2^t>2^a\sin\frac{\pi\epsilon}2$), from bounding below the denominator and
  above the numerator. This proves that $\phi(t)$ is as stated.
\end{proof}

Since $\phi(\log_2(m)) = -\log_2(m) - \frac12 + f_m + g_m$, we now get from
Proposition~\ref{prop:emfm} and the bound~\eqref{eq:gmbound}:
\begin{prop}\label{prop:emphi}
  The quantities $(e_m)$ defined by the recurrence~\eqref{eq:emrec}, or, more
  generally, defined for all real positive $m$ by
  Equation~\eqref{eq:emformula}, verify the asymptotic estimate, as
  $m\to+\infty$:
  \begin{equation*}
    e_m = \log_2(m) + \frac12 + \phi(\log_2(m)) + O(\frac{1}m),
  \end{equation*}
  where $\phi(t)$ is the function defined by Equation~\eqref{eq:phi}. It is
  $1$-periodic and analytic for $|\Im t |<\frac\pi{2\log(2)}$.
\end{prop}
We prove in the next section that $\phi(t)$ has zero mean.

\section{A Fourier series}\label{sec:fourier}

We compute the Fourier series of the $1$-periodic function $\phi(t)$ from
Equation~\eqref{eq:phi}.  Let $L$ be some positive integer. Let $n\in \ZZ$.  We start with
\begin{gather*}
\int_0^1 \e^{-2\pi i nt}\sum_{l=-\infty}^L\frac{2^{t-l}}{\e^{2^{t-l}}-1}\dt
=\int_0^1 \e^{-2\pi i nt}\sum_{l=-L}^\infty\frac{2^{t+l}}{\e^{2^{t+l}}-1}\dt
\\=\int_{-L}^{+\infty} \e^{-2\pi i nt}\frac{2^{t}}{\e^{2^{t}}-1}\dt
=\frac1{\log 2}\int_{2^{-L}}^\infty \e^{-2\pi i n\log_2(x)}\frac{\dx}{\e^x-1}
\\
=\frac1{\log 2}\int_{1}^\infty x^{-2\pi i n(\log 2)^{-1}}\frac{\dx}{\e^x-1}
+\frac1{\log 2}\int_{2^{-L}}^1  x^{-2\pi i n(\log 2)^{-1}}\Bigl(\frac{1}{\e^x-1}-\frac1x\Bigr)\dx
\\ + \frac1{\log 2}
\begin{cases}
  (-2\pi i n(\log2)^{-1})^{-1}\biggl[x^{-2\pi i n(\log 2)^{-1}}\biggr]_{2^{-L}}^1=0&(n\neq 0),
\\
  L\log(2)&(n=0).
\end{cases}
\end{gather*}
Let us first handle $n=0$.  We thus have:
\begin{align*}
  c_0(\phi)\coloneq\int_0^1 \phi(t)dt &=
    \lim_{L\to\infty} \biggl(-\frac12 - \frac12 + (L+1) 
    - \frac1{\log 2} \int_{2^{-L}}^\infty \frac{\dx}{\e^x-1}\biggr)
    \\
&=\lim_{L\to\infty}\biggl(L
    - \frac1{\log 2} \bigl[\log(1 - \e^{-x})\bigr]_{2^{-L}}^\infty\biggr)
    \\
&=\lim_{L\to\infty}\log_2\bigl(2^L(1 - \e^{-2^{-L}})\bigr) = 0.
\end{align*}
Suppose now $n\neq0$. We obtain
\begin{align*}
  c_n(\phi)&\coloneq\int_0^1\e^{-2\pi i nt}\phi(t)dt = \frac{1}{2\pi in}
-\frac1{\log2}
\begin{aligned}[t]
  \biggl(&\int_{1}^\infty x^{-2\pi i n(\log 2)^{-1}}\frac{\dx}{\e^x-1}
  \\&+\int_{0}^1  x^{-2\pi i n(\log 2)^{-1}}\Bigl(\frac{1}{\e^x-1}-\frac1x\Bigr)\dx
\biggr)
\end{aligned}
\\
&=\frac{1}{2\pi in}
-\frac1{\log2} \lim_{h\to0^+}
\begin{aligned}[t]
  \biggl(&\int_{1}^\infty x^{-2\pi i n(\log 2)^{-1}+h}\frac{\dx}{\e^x-1}
  \\&+\int_{0}^1  x^{-2\pi i n(\log 2)^{-1}+h}\Bigl(\frac{1}{\e^x-1}-\frac1x\Bigr)\dx
\biggr)
\end{aligned}
\\
&=
\frac{1}{2\pi in}
-\frac1{\log2} \lim_{h\to0^+}
  \biggl(\int_{0}^\infty x^{-2\pi i n(\log 2)^{-1}+h}\frac{\dx}{\e^x-1}
  -\int_0^1  x^{-2\pi i n(\log 2)^{-1}+h-1}\dx
\biggr)
\\
&=
\frac{1}{2\pi in}
-\frac1{\log2}
  \biggl(\Gamma\bigl(1-2\pi i n(\log 2)^{-1}\bigr)
         \zeta\bigl(1-2\pi i n(\log 2)^{-1}\bigr)
  -\frac{1}{-2\pi i n(\log2)^{-1}}
\biggr)
\\
&= -(\log 2)^{-1}\Gamma\bigl(1-\frac{2\pi i n}{\log 2}\bigr)
                \zeta\bigl(1-\frac{2\pi i n}{\log 2}\bigr).
\end{align*}
We have used the basic Riemann formula $\Gamma(s)\zeta(s)=\int_{0}^\infty
x^{s-1}(\e^x-1)^{-1}\dx$ for $\Re s>1$. Using the usual bound
$|\zeta(1+it)|=O_{t\to\pm\infty}(\log |t|)$ (\cite[Thm.\@ 3.5]{titchmarsh}), and
$|\Gamma(1+it)|\sim \sqrt{2\pi |t|}\exp(-\frac\pi2|t|)$ from Stirling formula
in the complex domain (or more expediently from the reflection formula
$\Gamma(s)\Gamma(1-s)=\pi/\sin(\pi s)$),
we recover the analyticity of $\phi(t)$ in the strip $|\Im t|<\pi/\log(4)$
which was stated in Proposition~\ref{prop:phiana}.


The proof of Theorem~\ref{thm:em} is complete.

\section{Ratios of Gamma functions}

Note that no statement such as $\Gamma(z+a)z^{-a}/\Gamma(z) = 1 +
O(|a|^2/|z|)$ for $|z|\to\infty$ holds if no conditions are set on $a$ (take
$z=a=n$, for $n\to\infty$).  The statement is true for a fixed $a$, or for a
bounded one (it must also be bounded away from zero).  The next Theorem
explains that the result holds if $|a|^2/|z| = O(1)$ (with rather
$O(|a|(1+|a|)/|z|)$ as first order error bound to account for small $a$'s).
We handle more generally the ratio $\Gamma(z+a)/\Gamma(z+b)$ and establish the
full asymptotic under the condition $\max(|a|^2,|b|^2) = O(|z|)$, for $z$ in
some angular sector avoiding the negative real axis; but we can speak really
of an \emph{asymptotic expansion} (of some sort) in general only if
$\max(|a|^2,|b|^2) = o(|z|)$ (the case $a=b+k$ for some fixed integer $k$
shows that the latter is only a sufficient condition, though).

As the matter was handled for fixed (or bounded) $a$ and $b$ in a well-known
1951 paper by \textsc{Tricomi}-\textsc{Erdélyi} \cite{tricomi1951}, we use
their notation with $\alpha$ and $\beta$.  The validity under the more general
condition which we consider here is presumably well-known, but the proof is
not quite trivial (it applies general principles that can generalize), and our
search through the literature brought no matches. It is not to be found in
\cite[\S2.11]{luke1969volI} nor in \cite[\S4.5]{olver1997}, despite the fact
that both discuss in detail the ratio of two Gamma functions; but for fixed
parameters. Reference \cite[\S5.11(iii)]{NIST:DLMF} also explicitly says that
the parameters are constants.
\begin{theo}\label{thm:gammaratio}
  Let $c_1>0$.  Let $0<\eta<\pi$.  There exists $A>0$, such that for every
  positive integer $J$ there exists a constant $C_J$ such that for every
  complex number $z$ verifying $|z|\geq A$ and $|\Arg z|\leq\pi - \eta$, and
  for every pair of complex numbers $(\alpha, \beta)$ verifying the condition
  \begin{equation*}
    \max(|\alpha|^2,|\beta|^2)\leq c_1 |z|,
  \end{equation*}
  the following inequality holds:
  \begin{equation}\label{eq:Rasymp}
    \left|\frac{\Gamma(z+\alpha)z^{\beta-\alpha}}{\Gamma(z+\beta)}
      - \sum_{0\leq j<J}\frac{Q_j(\alpha,\beta)}{z^j}\right|
    \leq C_J\frac{|\alpha-\beta|\max(1, |\alpha|, |\beta|)^{2J-1}}{|z|^J}\,,
  \end{equation}
  where $Q_0=1$, and $Q_j$ for $j>0$
  is a polynomial in $(\alpha,\beta)$.

  The $Q_j$'s are determined uniquely by the fact that $Q_0=1$,
  $Q_j(\alpha,\alpha)=0$ for $j>0$ and by either one of the following sets of
  recurrences:
  \begin{gather}
\label{eq:recQalpha}
\forall j>0, \quad Q_j(\alpha+1,\beta)-Q_j(\alpha,\beta) = \alpha
Q_{j-1}(\alpha,\beta)
\\
\label{eq:recQbeta}
\forall j>0, \quad Q_j(\alpha,\beta+1)-Q_j(\alpha,\beta) = -\beta
Q_{j-1}(\alpha,\beta+1).
  \end{gather}
$Q_j$ is of total degree $2j$.
\end{theo}
\begin{proof}
  Assuming we have proven the inequalities, they provide for a fixed pair
  $(\alpha,\beta)$ an asymptotic expansion as $|z|\to\infty$ with $z$ in the
  angular sector.  So the coefficients are unique.  It will be a corollary to
  the proof that $Q_0=1$, and that the $Q_j$'s are polynomials of total
  degrees in $(\alpha,\beta)$ at most $2j$.  The identities
  \begin{align*}
    \frac{\Gamma(z+\alpha+1)z^{\beta-\alpha-1}}{\Gamma(z+\beta)} &= \bigl(1 +
    \frac{\alpha}z\bigr)\frac{\Gamma(z+\alpha)z^{\beta-\alpha}}{\Gamma(z+\beta)}
    \\
    \frac{\Gamma(z+\alpha)z^{\beta-\alpha}}{\Gamma(z+\beta)} &= \bigl(1 +
    \frac{\beta}z\bigr)
    \frac{\Gamma(z+\alpha)z^{\beta-\alpha+1}}{\Gamma(z+\beta+1)}
  \end{align*}
  and unicity of asymptotic expansions imply the Equations
  \eqref{eq:recQalpha} and \eqref{eq:recQbeta}. Besides, again by unicity of
  asymptotic expansions, the $Q_j$'s for $j>0$
  must vanish for $\alpha=\beta$.  Equation \eqref{eq:recQalpha} determines
  $Q_j$, up to an additive constant depending on $\beta$ only, once $Q_{j-1}$ is known.
  As
  $Q_j$ has to vanish for $\alpha=\beta$, it is uniquely determined. The same
  applies with the recurrence \eqref{eq:recQbeta}.

  The statement on the total degrees being exactly $2j$ follows from the
  recurrences (see also the remarks located after the proof, on the relation
  with the generalized Bernoulli polynomials).

  Let us define $\cA = \{w\in \CC, |w|\geq 2, |\Arg w|\leq \pi- \eta\}$, which
  we call the \emph{small angular sector} (which is indented) and $\cB =
  \{w\in \CC, |w|\geq 1, |\Arg w|\leq \pi -\frac12\eta\}$, which we call the
  \emph{large} (indented) \emph{angular sector}.  For $z\in \cA$, we let $d_z$
  be the distance from $z$ to the complement of $\cB$, so that the closed disk
  centered at $z$ of radius $d_z$ is contained in the larger (closed) sector.
  There exists $\kappa>0$ such that $d_z\geq \kappa |z|$.  There is $A\geq 2$
  such that for $z\in \cA$, if $|z|\geq A$ and $|w|^2\leq c_1|z|$, then
  $|w|\leq \min(\frac12,\kappa)|z|$. One takes $A = \min\bigl(2,
  c_1/\min(1/4,\kappa^2)\bigr)$.  We note that then $z+w$ belongs to the
  larger sector (as well as the full segment from $z$ to $z+w$), and that
  $|w|\leq \frac12|z|$.

  We shall use the Stirling asymptotic series, as proven in
  \cite[\S12.33, \S13.6]{whittakerwatson4thed}, and in \cite[1.18]{erdelyiI}
  (for the latter, as consequence of the \emph{Watson's lemma} applied to a
  certain loop integral representation \cite[1.9(5)]{erdelyiI}; cf.\@
  \cite[\S1.4, \S2.11]{luke1969volI}, \cite[\S5.1]{olver1997},
  \cite[\S I.5]{wong2001}).
  There holds in the larger sector $\cB$, for some rational numbers $c_j$ related
  to Bernoulli numbers, and which we do not need to know explicitly for our aims:
  \begin{align}\label{eq:stirling}
    \log \Gamma(z) &= (z - \frac12) \Log z - z + \frac12\log(2\pi) +
    \sum_{j=1}^{J-1} c_j z^{-j} + r_J(z)
\\\notag
  r_J(z) &= O_{|z|\to\infty}(|z|^{-J})
  \end{align}
  On the right-hand side of Equation \eqref{eq:stirling}, we use the principal
  branch of the logarithm.  The left-hand side is the determination of $\log
  \Gamma(z)$ which is real for $z>0$, i.e.\@ it is the one vanishing at $z=1$.
  We pick $z$ in the smaller angular sector $|\Arg z| \leq \pi - \eta$, as in
  the statement of the Theorem, and verifying $|z|\geq A$ where $A\geq 2$ was
  specified earlier, so that if $\max(|\alpha|^2,|\beta|^2)\leq c_1|z|$, then
  both $z+\alpha$ and $z+\beta$ are in a closed disk centered at $z$ entirely
  contained in the larger sector $\cB$ and the asymptotic expansion applies
  for $z+\alpha$ (in inverse powers of $z+\alpha$) and at $z+\beta$ (in
  inverse powers of $z+\beta$).

  The remainder term in Equation \eqref{eq:stirling} verifies $r_J(z) =
  O(|z|^{-J})$. Hence (using this in a larger sector and the Cauchy formula to
  compute derivatives as contour integrals)
  we also have $r_J'(z)= O(|z|^{-J})$. But applying this to $J+1$ we get
  $r_{J+1}'(z) = O(|z|^{-J-1})$ and, comparing with $r_J'(z)$, we find for the
  latter that it also is $O(|z|^{-J-1})$ (self-improving stops here...).  As
  any point $z'$ of the segment $\Iff{z+\alpha}{z+\beta}$ is such that
  $|z'-z|\leq \frac12|z|$, there holds $|z'|\geq \frac12|z|$ and we have
  uniformly $|r_J'(z')|=O(|z'|^{-J-1}) = O(|z|^{-J-1})$ on that segment, so
  $r_J(z+\alpha)-r_J(z+\beta) = (\alpha-\beta)O(|z|^{-J-1})$.
   Of course, this argument would become unnecessary if we were using 
the well-known explicit form of $r_J(z)$ as an integral remainder with a Bernoulli polynomial (\cite[4.1]{olver1997}, \cite[p.\@ 52]{wong2001}).

  The function $\phi(w) = (w-\frac12)\Log w -w$ is analytic and its Taylor
  series at $z$ has a radius of convergence equal to $|z|$. The derivative is
  $\phi'(w) = \Log w - 1/(2w)$.  The Taylor series represents the function in
  the connected component of $z$ inside $\{z', |z'-z|<|z|,
  z'\not\in(-\infty,0]\}$.  Hence, as
  $\Iff{z}{z+\alpha}$ is entirely contained in that
  connected component, there holds (recalling also $|\alpha|\leq \frac12|z|$):
  \begin{align*}
    \phi(z+\alpha) &= \phi(z) + \alpha \Log z - \frac{\alpha}{2z} + \sum_{j=2}^\infty
    (-1)^j \alpha^j \Bigl(\frac{1}{j(j-1)z^{j-1}} + \frac{1}{2j z^j}\Bigr)
\\
&=
\phi(z) + \alpha \Log z
     + \sum_{j=1}^\infty (-1)^{j}\frac{(j+1)\alpha^j-
       2\alpha^{j+1}}{2j(j+1)z^j}\,.
\end{align*}
Let
\begin{equation*}
  M = \max(1, |\alpha|, |\beta|).
\end{equation*}
We know that $M\leq \frac12|z|$.
We consider $\phi(z+\alpha)-\phi(z+\beta)$, factorizing out $\alpha-\beta$. We estimate
for $j\geq J\geq 1$:
\begin{equation*}
  \left|\frac{(j+1)\frac{\alpha^j-\beta^j}{\alpha-\beta}-2\frac{\alpha^{j+1}-\beta^{j+1}}
     {\alpha-\beta}}{2j(j+1)z^j}\right|
   \leq \frac{jM^{j-1}+2M^j}{2j|z|^j} \leq \frac{3M^{j}}{2|z|^j}\leq \frac{3M^{J}}{2|z|^J}2^{-(j-J)}
\end{equation*}
Summing for $j\geq J$ gives thus a value $\leq 3M^{J}|z|^{-J}$.  We thus obtain, with $t_j$ a
certain polynomial of total degree at most $j$, for $1\leq j \leq J-1$:
\begin{gather*}
 \phi(z+\alpha)-\phi(z+\beta) - (\alpha - \beta)\Log z 
= (\alpha-\beta)\Biggl(\sum_{j=1}^{J-1}\frac{t_j(\alpha,\beta)}{z^j} + O\Bigl(\frac{M^{J}}{|z|^J}
\Bigr)\Biggr)\,.
\end{gather*}
The earlier $(\alpha-\beta)O(|z|^{-J-1})$ bound for ${r_J(z+\alpha)-r_J(z+\beta)}$ is negligible in comparison to $(\alpha-\beta)
O\bigl(\frac{M^{J}}{|z|^J}\bigr)$ and collecting our estimates we have:
\begin{gather*}
\frac{\Gamma(z+\alpha)z^{\beta-\alpha}}{\Gamma(z+\beta)} = \e^{(\alpha-\beta)S}\,,
\\
\shortintertext{with}
S = \sum_{j=1}^{J-1}\frac{t_j(\alpha,\beta)}{z^j}
  + \sum_{j=1}^{J-1} c_j \frac{(z+\alpha)^{-j}-(z+\beta)^{-j}}{\alpha-\beta}
   + O\Bigl(\frac{M^{J}}{|z|^J}\Bigr)\,.
\end{gather*}
Using for each $j$ the binomial series to expand $(1+ \alpha z^{-1})^{-j}$ and
$(1+\beta z^{-1})^{-j}$ in inverse powers of $z$, in order to handle
$\frac{(z+\alpha)^{-j}-(z+\beta)^{-j}}{\alpha-\beta}$, and estimating as was
done previously the remainders of truncated series, using again $M = \max(1,
|\alpha|,|\beta|)\leq \frac12|z|$ to that purpose, we obtain:
\begin{equation*}
  S = \sum_{j=1}^{J-1}\frac{T_j(\alpha,\beta)}{z^j} + O\Bigl(\frac{M^{J}}{|z|^J}\Bigr)\,,
\end{equation*}
where $T_j$, $1\leq j\leq J-1$, is a polynomial of total degree at most $j$.
In particular with $J=1$, we obtain $S = O\bigl(\frac{M}{|z|}\bigr)$.  We note
for future reference that $S^J = O\bigl(\frac{M^J}{|z|^J}\bigr)$.  We also
note that the product $(\alpha-\beta)S$ is $O(M^2/|z|)$ hence $O(1)$.  Thus
with some implied constant depending only on $c_1$, $\eta$, and $J$:
\begin{align*}
  \e^{(\alpha-\beta) S} &= 1 + (\alpha-\beta) S + \frac{(\alpha-\beta)^2}2S^2 + \dots +
\frac{(\alpha-\beta)^{J-1}}{(J-1)!} S^{J-1} +
  O(|\alpha-\beta|^J S^J)
\\
&= 1 + \sum_{1\leq j<J} \frac{(\alpha-\beta)^{j}}{j!} S^{j} +
  O\Bigl(|\alpha-\beta|^JM^{J} |z|^{-J}\Bigr)
\\
&= 1 + \sum_{1\leq j<J} \frac{(\alpha-\beta)^{j}}{j!} S^{j} +
  O\Bigl(|\alpha-\beta|M^{2J-1} |z|^{-J}\Bigr)\,.
\end{align*}
Consider now the various contributions originating in some given
power $S^j$, $j<J$. Those not involving the big-$O$ term in $S$ will be of the type
\begin{equation*}
  (j!)^{-1}(\alpha-\beta)^{j}T_{i_1}T_{i_2}... T_{i_j} z^{-i_1 - i_2 - \dots - i_j}.
\end{equation*}
If $d=i_1+i_2+\dots+i_j<J$, we get a contribution $c_d/z^d$ where $c_d$ is a
polynomial of total degree $j+d$ in $(\alpha,\beta)$.  This is at most $2d$ as all $i_p$'s are
positive.
Suppose now that $d=i_1+i_2+\dots +
i_j = J + q$ with $q\geq0$. Then, bounding each $T_{i_p}$, and recalling
$M=O(\sqrt{|z|})$ we get
\begin{equation*}
  O\Bigl(|\alpha-\beta|^{j}M^{J+q}|z|^{-J-q}\Bigr)
\subset O\Bigl(|\alpha-\beta|^{j}M^{J-q}|z|^{-J}\Bigr)
\end{equation*}
And $j+J-q\leq j+J<2J$.  So this is
$O(|\alpha-\beta|M^{2J-2}/|z|^J)$; this is even better
than the exponent $2J-1$, which thus originates solely from the
$O(|\alpha-\beta|^JS^J)$ contribution to $\exp((\alpha-\beta)S)$.

To handle the terms from $S^j$ involving one or more times the big-$O$ in $S$,
we use the notation $|T_J|= M^J$.  So we need to handle now
$|\alpha-\beta|^j|T_{i_1}|\allowbreak|T_{i_2}|\allowbreak...\allowbreak|T_{i_k}|
|z|^{-i_1 - i_2 - \dots i_k}$, but with some of the indices being $J$. The
exact same reasoning as before in the case $d=J+q$, $q\geq0$, applies.  This
completes the proof.
\end{proof}

\begin{rema}
  These asymptotic expansions are, naturally, well-known.
  \textsc{N\"{o}rlund} obtained in \cite[p.\@ 11]{norlund1910} (which was his
  Dissertation) an asymptotic expansion $\Gamma(z)z^\beta/\Gamma(z+\beta)\sim
  \sum_{j=0}^\infty c_j(\beta)z^{-j}$ for $z$ going to infinity in an angular
  sector avoiding the negative real axis, but (presumably in view of his later
  papers) this is for a \emph{fixed}, or at least \emph{bounded} $\beta$.  See
  also \cite[p.\@ 367]{norlund1914} and \cite[p.\@ 43]{norlund1961} where
  again $\beta$ appears to be fixed or bounded.  \textsc{Tricomi}
  and \textsc{Erdélyi} (\cite{tricomi1951}) generalized \textsc{N\"{o}rlund}
  investigations and considered the asymptotics of
  $\Gamma(z+\alpha)/\Gamma(z+\beta)$, popularizing in passing a tool called
  \emph{Watson's lemma} (see an account of this in \cite[\S4.5]{olver1997}).
  Again, the parameters $\alpha$ and $\beta$ there are bounded (this is
  explicitly stated at the start of the second paragraph of their paper, and
  again in an example at the end of the paper).
  See also \cite[\S2.11]{luke1969volI} and \cite[\S5.1]{olver1997}.
  A related expansion was given by \textsc{Fields} \cite{fields1966}.
\end{rema}
\begin{rema}
  In the works by the above quoted authors, it is shown (\cite[Eq.\@
  (19)]{tricomi1951}) that the $Q_j$'s are related to the \emph{generalized
    Bernoulli numbers and polynomials}, which where studied in particular by
  \textsc{N\"{o}rlund} (\cite{norlund1914, norlund1924}). There holds
  \begin{equation*}
    Q_j(\alpha,\beta)=(-1)^j\frac{(\beta-\alpha)_j}{j!}B_j^{(\alpha-\beta+1)}(\alpha)\,,
  \end{equation*}
  where $(\beta-\alpha)_j$ is the ascending partial
  factorial, and (\cite{norlund1924, norlund1961}):
  \begin{equation*}
    \frac{t^\rho\e^{xt}}{(\e^t-1)^\rho} 
= \sum_{j\geq0} B_j^{(\rho)}(x)\frac{t^j}{j!}\qquad(|t|<2\pi)\,.
  \end{equation*}
\end{rema}
\begin{rema}\label{rem:Pja}
  In the sequel we will encounter a situation where we need $P_{j,a}(t) =
  Q_j(a, a +t)$ where $a$ is fixed and $t$ is a quantity of interest with
  large imaginary part.  Their
  recurrence, from Equation \eqref{eq:recQbeta}, reads:
\begin{equation*}
  P_{j,a}(t+1) - P_{j,a}(t) = - (t+a) P_{j-1,a}(t+1),
\end{equation*}
with initial conditions $P_{0,a}=1$, and $P_{j,a}(0)=0$ for $j>0$.  
In terms of generalized Bernoulli polynomials (\cite{norlund1924})
$P_{j,a}(t)=(-1)^j\frac{(t)_j}{j!}B_{j}^{(1-t)}(a)$.
\end{rema}

\section{Contour integrals}

Let us recall the following contour integral, which goes under the name of
\emph{Rice formula} (already found in \cite[Ch.\@ 8]{norlund1924}; see \cite[Ex.\@
5.2.2-54, p.\@ 138]{knuth}, \cite[Eq.\@ (2)]{flajosedge1995}, \cite[Thm.\@
8.20]{szpbook2001}, \cite[\emph{N\"{o}rlund-Rice integrals}, p.\@
238]{flajosedgebook}):
\begin{lem}
  Let $f$ be analytic on some domain containing the segment $[i,k]$ where
  $0\leq i\leq k\leq n$ are integers. Let $\cC$ be a counter-clockwise Jordan
  contour which is included, together with the region it bounds, in that
  domain and which goes through none of the integers $j\in \Iff{0}{n}$.
  Assume further that the integer points in $\Iff{0}{n}$ which are in the
  region bounded by $\cC$ are $j=i$, \dots, $k$.  Then (the letter $i$ being
  used for two purposes):
  \begin{equation*}
    \sum_{j=i}^k(-1)^j\binom{n}{j}f(j) = \frac1{2\pi i}\int_{\cC}\frac{(-1)^nn! f(z)}{z(z-1)\dots(z-n)}\dz.
  \end{equation*}
\end{lem}
\begin{proof}
  The residue at $j\in\Iff{i}{k}$ is $\frac{(-1)^nn!f(j)}{j(j-1)\dots
    1\cdot(-1)\dots(-(n-j))} = (-1)^j\binom{n}{j}f(j)$.
\end{proof}
Suppose that $f$ is analytic in the full strip $i-\frac34<\Re z<k+\frac34$ and
has polynomial growth there as $|\Im z|\to\infty$.  We apply the Lemma with a
rectangular contour with vertices at $k+\frac12-iT, k+\frac12+iT,
i-\frac12+iT, i-\frac12-iT$ and then let $T\to+\infty$. This gives, for $n$
large enough:
\begin{equation*}
  \frac1{2\pi i}\left(\int_{k+\frac12-i\infty}^{k+\frac12+i\infty} - \int_{i-\frac12-i\infty}^{i-\frac12+i\infty}\right)\frac{(-1)^nn! f(z)}{z(z-1)\dots(z-n)}\dz.
\end{equation*}
Let us suppose now that $f$ is actually analytic in the entire half-plane
$\Re z > i-\frac34$ and has polynomial growth in this half-plane as
$|z|\to\infty$, and that $k=n$, so that there are no singularities to the
right of $\Re z = k+\frac12$.  We can then replace the
segment $\Iff{k+\frac12-iT}{k+\frac12+iT}$ by the half-circle with center at
$k+\frac12$ joining these extremities.  Using the polynomial growth condition
on this half-circle and the two remaining half-lines, and letting
$T\to+\infty$, we obtain that for large enough $n$ the integral on $\Re z =
n+\frac12$ vanishes.  So under these conditions, we obtain
\begin{equation*}
  \sum_{j=i}^n(-1)^j\binom{n}{j}f(j) = 
- \frac1{2\pi i}\int_{i-\frac12-i\infty}^{i-\frac12+i\infty}
  \frac{(-1)^nn! f(z)}{z(z-1)\dots(z-n)}\dz.
\end{equation*}
This remains valid with a summation up to $k<n$ if $f$ vanishes at $k+1$, \dots, $n$.

A particularly useful case is given by the exponential function $z\mapsto x^z$,
$x>0$ which, depending on whether $x>1$ or $x<1$ is small either for $\Re
z\to-\infty$ or for $\Re z\to+\infty$, and, for $x=1$ is bounded, actually
constant, in the entire complex plane.  Choosing $i=0$, $k=n$, and $f(z)=x^z$
with $0<x\leq 1$, we obtain in this special case:
\begin{equation*}
   \sum_{j=0}^n(-1)^j\binom{n}{j}x^j =
  \frac{-1}{2\pi i}\int_{-\frac12-i\infty}^{-\frac12+i\infty}
  \frac{(-1)^nn! x^z}{z(z-1)\dots(z-n)}\dz.
\end{equation*}
This is valid for any $n>0$ (even if $x=1$).  Let us
replace $z$ by $-z$. We then obtain:
\begin{equation*}
   x\leq 1\implies (1 -x)^n = 
  \frac{1}{2\pi i}\int_{\frac12-i\infty}^{\frac12+i\infty}
  \frac{n! x^{-z}}{z(z+1)\dots(z+n)}\dz.
\end{equation*}
The integral is absolutely convergent. In analytic number theory, it is
important to also consider the case $n=0$ which gives a semi-convergent
integral (if done symmetrically to real axis), but we don't need this here.
For $x>1$, we can shift the line of integration to larger real parts, we
conclude that the integral represents zero. We have established the following
\emph{Mellin-Perron} formula
\begin{equation}\label{eq:mellin}
   \frac{1}{2\pi i}\int_{\frac12-i\infty}^{\frac12+i\infty}
  \frac{n! x^{-z}}{z(z+1)\dots(z+n)}\dz =
  \begin{cases}
    (1-x)^n& 0<x\leq 1,\\
    0 &1\leq x\,.
  \end{cases}
\end{equation}
Such identity is the inversion formula (as in \emph{Fourier inversion}; use
the change of variable $x=\e^{u}$)
associated with the value of
$\int_0^1(1 - x)^n x^{z-1}\dx$, which itself is related to the identity:
\begin{equation*}
   \sum_{k=0}^n \frac{(-1)^k \binom{n}{k}}{z+k} = \frac{n!}{z(z+1)\cdots(z+n)}\,.
\end{equation*}
In analytic number theory, \emph{Perron formulas} allow to
investigate the summatory function $\sum_{k<x} a_k +\frac12\delta_{x}(k)a_k$,
or smoother variants such as
$\sum_{k\leq x} (1 - k/x)^n a_k$, from the properties of the Dirichlet
series $\sum a_k k^{-z}$, and are often formulated as integrals
for $|\Im z|\leq T$ with an error bound (cf.\@ \cite[3.19]{titchmarsh},
\cite[II.2]{tenenbaumGSM}).  But we do not need such refinements for our
purposes.

References on this
topic in the field of combinatorics and theoretical computer science include
\cite{flajosedge1995}, \cite[\S 7.5.3]{szpbook2001}, \cite[App.\@
B]{flajosedgebook}, and of course \cite[5.2.2]{knuth}.

\section{Mellin transforms involving the Riemann zeta function}

Using Equation \eqref{eq:mellin} as the crucial tool, we
express $e_m$ analytically.  Equation \eqref{eq:empm} gives
$e_m$ in terms of the moments $p_m(x) =  x^{-1}\sum_{0\leq k <x} (x^{-1}k)^m$
for $x=2^l$, $l\geq0$.
We observe that for $m\geq0$ and with $x$ a positive \emph{integer}:
\begin{equation*}
  x^{-1}\sum_{0\leq k <x} (x^{-1}k)^m = x^{-1}\sum_{0<k\leq x}(1 - x^{-1}k)^m \,.
\end{equation*}
Dropping now the condition that $x$ is an integer, we reformulate using
Equation \eqref{eq:mellin} the above right-hand side, for $m>0$, as a formally
infinite sum, which has in truth only finitely many non-zero terms:
\begin{equation*}
  x^{-1}\sum_{0<k\leq x}(1 - x^{-1}k)^m  =
\sum_{k=1}^\infty \frac{x^{-1}}{2\pi i}\int_{\frac12-i\infty}^{\frac12+i\infty}
  \frac{m! x^{z}k^{-z}}{z(z+1)\dots(z+m)}\dz\,.
\end{equation*}
We have assumed $m\geq1$ here so that Equation \eqref{eq:mellin} applies as
is, and the integrand, for each term,  is absolutely integrable on the $\Re z=\frac12$ line.
In order to move the summation inside the integrand, we first shift (for each
term) the line of integration to $\Re z = 2$, which keeps all terms unchanged,
then we can permute. Thus, for any $x>0$ (integer or not):
\begin{equation*}
  x^{-1}\sum_{0<k\leq x}(1 - x^{-1}k)^m  = \frac{x^{-1}}{2\pi i}\int_{2-i\infty}^{2+i\infty}
  \frac{m! x^{z}\zeta(z)}{z(z+1)\dots(z+m)}\dz\,.
\end{equation*}
It is known that $\zeta(z)=O(|\Im z|^{\frac14+\epsilon})$ for $\frac12\leq \Re
z\leq 2$ (and $|z-1|\geq \frac14$), for any $\epsilon>0$: the Lindelöf
conjecture says that we can even remove the $\frac14$.  Exponents $<\frac14$
such as $\frac16$ are known to hold (\emph{sub-convexity bounds},
\cite[Thm.\@ 5.5]{titchmarsh}).  For $m$ large enough, we  only need
the polynomial growth property of $\zeta(s)$ on vertical strips of finite
width, which is elementarily provable, but the $O(|\Im
z|^{\frac14+\epsilon})$ allows to shift the line of
integration back to $\Re z=\frac12$, for every $m\geq1$.
This process keeps an absolutely integrable integrand, and picks up a residue at $z=1$:
\begin{equation*}
  x^{-1}\sum_{0<k\leq x}(1 - x^{-1}k)^m  = \frac{1}{m+1} + 
   \frac{1}{2\pi i}\int_{\frac12-i\infty}^{\frac12 +i\infty}
  \frac{m! x^{z-1}\zeta(z)}{z(z+1)\dots(z+m)}\dz\,.
\end{equation*}
This is excellent news, as it provides, for any positive real $x$, and positive integer $m$, the formula:
\begin{equation}
  \label{eq:pmmellin}
  1 - (m+1)x^{-1}\sum_{0<k\leq x}(1 - x^{-1}k)^m = \frac{-1}{2\pi i}\int_{\frac12-i\infty}^{\frac12 +i\infty}
  \frac{(m+1)! x^{z-1}\zeta(z)}{z(z+1)\dots(z+m)}\dz\,,
\end{equation}
which, combined with Equation \eqref{eq:empm} gives:
\begin{equation*}
  e_m = - \sum_{l=0}^\infty \frac{1}{2\pi i}\int_{\frac12-i\infty}^{\frac12 +i\infty}
  \frac{(m+1)! (2^l)^{z-1}\zeta(z)}{z(z+1)\dots(z+m)}\dz\,.
\end{equation*}
As on the line of integration $\Re(z-1) = - \frac12$, there is no difficulty in permuting the summation and the integral, and we obtain the exact representation:
\begin{lem}
  For positive integer $m$, there holds (mind the sense of integration
  downwards from $\frac12+i\infty$ to $\frac12-i\infty$):
  \begin{equation}
    \label{eq:emmellin}
     e_m = \frac{1}{2\pi i}\int_{\frac12+i\infty}^{\frac12 -i\infty}
  \frac{(m+1)! \zeta(z)}{z(z+1)\dots(z+m)}\frac{2^{1-z}}{2^{1-z} - 1}\dz\,.
  \end{equation}
\end{lem}

\section{Complete asymptotics for \texorpdfstring{$e_m$}{em}}

Let us from now on use the notation (only for $k$ a non zero relative integer):
  \begin{equation*}
    \chi_k = \frac{2\pi i k}{\log 2}\,.
  \end{equation*}
  We shift in Equation \eqref{eq:emmellin} the line of integration to
  $\Re(z)=2$, picking up residues at $z=1$ (double pole) and at the $z = 1 +
  \chi_k$'s, $k\in \ZZ\setminus\{0\}$.  To be more specific, we consider a
  rectangular contour with sides parallel to the axes and having its two horizontal
  segments passing midway between $1+\chi_{\pm K}$ and $1+\chi_{\pm(K+1)}$. We
  then let
  let $K\to \infty$.  Observe that the $2^{1-z}$ is bounded, as well as is
  inverse, uniformly on such contours.  Once this is done, we can move the
  line of integration farther and farther to the right, and as $\zeta(z)$ is
  bounded for $\Re z\geq 2$, we see that this contribution actually vanishes.
  So $e_m$ is equal to the sum of the residues, which we pick up positively,
  because of the downwards direction of integration when still on the left of
  the poles:
  \begin{equation*}
    e_m =
    \begin{aligned}[t]
      &\Res_{z=1} \frac{(m+1)! \zeta(z)}{z(z+1)\dots(z+m)}\frac{2^{1-z}}{2^{1-z} - 1}
\\
      &-(\log 2)^{-1}\sum_{\substack{k\in\ZZ\\k\neq 0}}
\frac{(m+1)!\; \zeta(1+\chi_k)}{(1+\chi_k)(2+\chi_k)\dots(m+1+\chi_k)}\,.
    \end{aligned}
  \end{equation*}
  Using $\zeta(1+h) = \frac1{h}\bigl(1 + \gamma h + O(h^2)\bigr)$, we compute
  that the residue at $z=1$ is equal to $(\log 2)^{-1} H_{m+1} + \frac12 -
  (\log 2)^{-1}\gamma$.  Hence the following theorem:
  \begin{theo}\label{thm:emexact}
    For any positive integer $m$, there holds:
    \begin{equation}
     \label{eq:emexact}
      e_m = \frac{H_{m+1}}{\log 2} + \frac12 - \frac{\gamma}{\log 2} 
           -(\log 2)^{-1}\sum_{\substack{k\in\ZZ\\k\neq 0}}r_k(m)\,,
\end{equation}
with
\vspace*{-\baselineskip}
\begin{align}
\label{eq:rkm1}
     r_k(m) &= \frac{(m+1)!\; \zeta(1+\chi_k)}{(1+\chi_k)(2+\chi_k)\dots(m+1+\chi_k)}
\\\label{eq:rkm2}
          &= \frac{\Gamma(m+2)}{\Gamma(m+2+\chi_k)}\Gamma(1+\chi_k)\zeta(1+\chi_k).
\end{align}
  \end{theo}
  \begin{rema}
    As we mentioned in the first section of the paper, this result
    appeared already in \cite[p.\@ 113]{flajosedge1995}.%
\footnote{The quantity $V_n$ there is $1 - e_{n-1}$; 
the bottom of the page is actually about $-V_n$ not $V_n$.}
  \end{rema}
  Approximating, formally at first, the ratio of Gamma functions by
  $(m+2)^{-\chi_k} = \exp\bigl(-2\pi i k \log_2(m+2)\bigr)$, the oscillating
  term is here provided as a Fourier series in $\log_2(m+2)$, not $\log_2(m)$.
  But, the function $\phi$ of Equation \eqref{eq:thm2phiseries} being smooth
  and periodic, $\phi(\log_2(m+2)) = \phi(\log_2(m)) + O(m^{-1})$, so there is
  no contradiction, fortunately.

  Examining Equation \eqref{eq:rkm2} under the light of Theorem
  \ref{thm:gammaratio} we have the choice of which quantity we want to expand
  in inverse powers of, with periodic decorations.  The simplest choice is in
  inverse powers of $m+2$.  But we can choose to expand in inverse powers of
  $m$, or of $m+1$, or \dots, as we like. We will start with $m+2$ and
  explain later how to do otherwise.

  First, we use Equation \eqref{eq:rkm1} for some sub-optimal but easy estimate on
  the series of residues:
  \begin{lem}
    For any positive integer $K$, there holds
    \begin{equation*}
      \sum_{\substack{k\in\ZZ\\|\chi_k|> \sqrt{m}}}|r_k(m)| = o(m^{-K})
    \end{equation*}
  \end{lem}
  \begin{proof}
    We note that even with $m=1$, the series of residues is absolutely
    convergent.  This follows from $\zeta(1+\chi_k) = O(\log(1+|k|))$
    (cf. \cite[Thm.\@ 3.5]{titchmarsh}).  For $m\geq 2K+1$, lower-bounding
    $|j+\chi_k|$ by $|\chi_k|$ if $j\leq 2K+2$ and by $j$ otherwise, we get:
    \begin{equation*}
      |r_k(m)|\leq |\zeta(1+\chi_k)|\frac{(2K+2)!}{|\chi_k|^{2K+2}}
        = O\Bigl(\frac{\log (|k|+1)}{|k|^{2K+2}}\Bigr)\,.
    \end{equation*}
    Let $k_0$ be the maximal positive integer such that $|\chi_{k_0}|\leq
    \sqrt{m}$ (we assume $K$ large enough so that $k_0$ exists for $m\geq
    2K+1$). Summing over the $k$'s with $|k|>k_0$ gives a
    value which is $O(\log(|k_0|)/k_0^{2K+1})$, hence it is $o(m^{-K})$.
  \end{proof}

  Recall that for $t$ real,
  $|\Gamma(1+it)|^2=\Gamma(1+it)\Gamma(1-it)=it\frac{\pi}{\sin(\pi
    it)}=\frac{\pi t}{\sinh (\pi t)}$, hence the Gamma function has
  exponential decrease at infinity on the line $\Re z = 1$.  Thus, there is
  some $c<1$ such that $|\Gamma(1+\chi_k)\zeta(1+\chi_k)|=O(c^k)$.  We now
  establish Theorem \ref{thm:fullasymp}.
\begin{proof}[Proof of Theorem \ref{thm:fullasymp}]
  Let $R(m) = \sum_{k\neq 0, |\chi_k|^2\leq m} r_k(m)$.  By the preceding
  lemma it differs from the sum over all non-zero relative integers $k$ by an
  error which is asymptotically smaller than any inverse power of $m$.
  According to Theorem \ref{thm:gammaratio}, applied with $\alpha=0$ and
  $\beta=\chi_k$, and writing $P_j(t) = Q_j(0,t)$ (see Remark \ref{rem:Pja}
  and recall that from its defining recurrence relation $\deg P_j = 2j$)
  we have, with some implied constant in the big-$O$ depending only on $J$, for
  $k$ such that $|\chi_k|^2\leq m$:
\begin{align*}
  r_k(m) &= \frac{\Gamma(m+2)(m+2)^{\chi_k}}{\Gamma(m+2+\chi_k)}
            (m+2)^{-\chi_k}\Gamma(1+\chi_k)\zeta(1+\chi_k)
\\
&= \left(\sum_{0\leq j<J}\frac{P_j(\chi_k)}{(m+2)^j} 
   + O\bigl(\frac{|\chi_k|^{2J}}{(m+2)^J}\bigr)\right)
    \cdot(m+2)^{-\chi_k}\Gamma(1+\chi_k)\zeta(1+\chi_k).
\end{align*}
Let us define for $j\geq0$:
\begin{equation*}
  S_j(m+2) = \sum_{k,|\chi_k|^2\leq m} P_j(\chi_k)\Gamma(1+\chi_k)\zeta(1+\chi_k)(m+2)^{-\chi_k}.
\end{equation*}
There holds, recalling $|\Gamma(1+\chi_k)\zeta(1+\chi_k)|=O(c^k)$ for some $c<1$:
\begin{align*}
  R(m) &= \sum_{0\leq j<J}\frac{S_j(m+2)}{(m+2)^j} + O\bigl(\frac{\sum_{k,
      |\chi_k|^2\leq m}|k|^{2J}c^k}{(m+2)^J}\bigr)
  \\
  &=\sum_{0\leq j<J}\frac{S_j(m+2)}{(m+2)^j} + O(\frac1{(m+2)^{J}})\,.
\end{align*}
There exists $\eta>0$ such that for each $j<J$:
\begin{equation*}
  \sum_{k, |\chi_k|^2>m} \Bigl|P_j(\chi_k)\Gamma(1+\chi_k)\zeta(1+\chi_k)\Bigr|
 = O\Bigl(\sum_{k, |\chi_k|^2>m} |k|^{2j}c^k\Bigr) = o(\e^{-\eta \sqrt{m}}).
\end{equation*}
Hence, it  is a fortiori
$O((m+2)^{-J})$, so up to moving the difference into the final big-$O$, we can
replace $S_j(m+2)$ with a full sum over $k\in \ZZ$, $k\neq0$.  This completes
(after replacing $k$ by $-k$) the proof of Theorem \ref{thm:fullasymp}, for
the case of inverse powers of $m+2$.

There remains to consider the second part of Theorem \ref{thm:fullasymp},
which involves an extra real parameter $a$.  For this, we write, for $m$
large enough:
\begin{equation*}
  \frac{\Gamma(m+2)}{\Gamma(m+2+\chi_k)} = 
   \frac{\Gamma(m+2-a+a)(m+2-a)^{\chi_k}}
        {\Gamma(m+2-a + a +\chi_k)}\e^{-2\pi i k \log_2(m+2-a)}\,.
\end{equation*}
With this modified starting point, we will use Theorem
\ref{thm:gammaratio} for $z=m+2-a\to+\infty$ with $\alpha=a$ and $\beta=a +
\chi_k$. All our previous steps go through similarly.  This gives the
asymptotic expansion Equation \eqref{eq:asympgen} in inverse powers of
$m+2-a$, as claimed in Theorem \ref{thm:fullasymp}, up to the change from $k$
to $-k$.
\end{proof}

\section{Back to the exercise 6.3-34}

Let us recall from Equation \eqref{eq:pmmellin}, replacing $m+1$ by $m$ and
thus assuming $m\geq2$:
\begin{equation*}
1 - m x^{-1}\sum_{0<k\leq x}(1 - x^{-1}k)^{m-1} 
= \frac{-1}{2\pi i}\int_{\frac12-i\infty}^{\frac12 +i\infty}
  \frac{m! x^{z-1}\zeta(z)}{z(z+1)\dots(z+m-1)}\dz
\end{equation*}
There holds $\zeta(z)=O(|z|^{1+\epsilon})$ for $-\frac12\leq \Re z\leq
\frac12$ and any given $\epsilon>0$ (\cite[\S5.1]{titchmarsh}). So, for $m\geq
3$, we can shift the integration contour to $\Re z = -\frac12$, picking up
a
residue (its opposite rather due to the $-1$ factor) at the simple pole at $z=0$, hence, $\frac {m}{2x}$ due
to $\zeta(0)=-\frac12$ (\cite[Eq.\@ (2.4.3)]{titchmarsh}):
\begin{equation*}
1 - m x^{-1}\sum_{0<k\leq x}(1 - x^{-1}k)^{m-1} =
   \frac{m}{2x} + \frac{-1}{2\pi i}\int_{-\frac12-i\infty}^{-\frac12 +i\infty}
  \frac{m! x^{z-1}\zeta(z)}{z(z+1)\dots(z+m-1)}\dz
\end{equation*}
Rearranging, we obtain
\begin{equation*}
  \sum_{0<k\leq x}(1 - x^{-1}k)^{m-1} - \frac{x}{m} + \frac12 = 
\frac{+1}{2\pi i}\int_{-\frac12-i\infty}^{-\frac12 +i\infty}
  \frac{(m-1)! x^{z}\zeta(z)}{z(z+1)\dots(z+m-1)}\dz\,.
\end{equation*}
The left-hand side, if taken for $x=2^l$, $l\geq1$, is \emph{exactly} the
quantity from Equation \eqref{eq:aml}.  As $|2^{z}| = 1/\sqrt{2}$, on the line
of integration, we can move
the summation over all those powers $x=2^l$ inside the integrand, and we
obtain the exact representation:%
\footnote{Mind that this $a_m$ is not exactly the $a_n$ from Equation (18) of
\cite[\S6.3]{knuth}.}
\begin{equation}
   \label{eq:ammellin}
a_m = \frac{1}{2\pi i}\int_{-\frac12-i\infty}^{-\frac12 +i\infty}
  \frac{(m-1)! \zeta(z)}{z(z+1)\dots(z+m-1)}\frac{\dz}{2^{-z} - 1}\,.
\end{equation}
We can now shift back to $\Re z=\frac12$, picking up the (opposite of the)
residues located at $z=0$ (where there is a double pole) and at $z=
\chi_k = 2\pi i k/\log (2)$.  And then we shift even further to the right, to
$\Re z=2$, with a contribution from the pole at $z=1$.  Once located on the
line $\Re z= 2$, the integral vanishes because $\zeta(z)$ is bounded for $\Re
z\geq 2$. We thus obtain the following
exact formula.

First
\begin{equation}
a_m = \frac{H_{m-1}}{2\log 2} + \frac{\zeta'(0)}{\log(2)} - \frac14 + \frac{2}m 
+ (\log 2)^{-1}\sum_{\substack{k\in\ZZ\\k\neq 0}}
\frac{(m-1)!\zeta(\chi_k)}{\chi_k(\chi_k+1)\dots(\chi_k+m-1)}\,,
\end{equation}
then, using $\zeta'(0) = -\frac12\log(2\pi)$ (\cite[Eq.\@ (2.4.5)]{titchmarsh}):
\begin{equation}
   \label{eq:amexact}
a_m = \frac{H_{m-1}}{2\log 2} - \frac{\log\pi}{2\log 2} - \frac34 + \frac2{m}
+ (\log 2)^{-1}\sum_{\substack{k\in\ZZ\\k\neq 0}}
\frac{\Gamma(m)\Gamma(\chi_k)\zeta(\chi_k)}{\Gamma(m+\chi_k)}.
\end{equation}
Using Theorem \ref{thm:gammaratio} as in the proof of Theorem
\ref{thm:fullasymp}, we would obtain via the very similar steps (details are
left to the reader):
\begin{prop}
  The quantity $\delta_m = a_m - \frac{H_{m-1}}{2\log 2} +
  \frac{\log\pi}{2\log 2} + \frac34 - \frac2{m}$ admits an asymptotic
  development to all orders in inverse powers of $m$, decorated by
  $1$-periodic functions of $\log_2(m)$:
  \begin{equation*}
    \delta_m \sim \frac1{\log 2}\sum_{j=0}^\infty \frac{\theta_j(\log_2(m))}{m^j}\,,
  \end{equation*}
with
\begin{equation*}
  \theta_j(t) = \sum_{\substack{k\in\ZZ\\k\neq0}}
  P_j(-\chi_k)\Gamma(-\chi_k)\zeta(-\chi_k)\e^{2\pi i k t}\,.
\end{equation*}
Here, $(P_j)$ is the sequence of polynomials already considered in Theorem
\ref{thm:fullasymp}, i.e.\@, $P_0=1$, $P_j(0)=0$ for $j>0$ and
$P_j(t+1)-P_j(t) = - t P_{j-1}(t+1)$.
\end{prop}

The symbol $\sim$ means that keeping only from $j=0$ to $j=J-1$ represents
$\delta_m$ with an error which is $O(m^{-J})$.  The periodic function
$\theta_0(t)$ is already to be found on page 727 of \cite{knuth}.  But it is
derived there from first finding the approximation which we have denoted $b_m$
in 
section \ref{sec:knuth}, and it is then $b_m$ which is represented by an
integral in the complex plane (see \emph{op.\@ cit.\@}, p.\@ 510).

\section{Miscellaneous concluding remarks}

Let us point out that there is an alternative formula for $f_m$, $m>0$,
which is more convenient numerically:
\begin{prop}\label{prop:fm}
  Let $f_m$ for real $m>0$ be defined by Equation~\eqref{eq:fm},i.e.\@
  \begin{equation*}
  f_m = \sum_{l=0}^\infty \bigl(1 - \frac{m2^{-l}}{\e^{m2^{-l}}-1}\bigr).
  \end{equation*}
  There holds:
  \begin{equation*}
    f_m  = \sum_{l=1}^\infty \frac{l 2^{-l}m}{\e^{2^{-l}m}+1}\,.
  \end{equation*}
\end{prop}
For the proof we will use the Lemma~\ref{lem:viete} below, which is undoubtedly very
well-know, but which we did not locate in standard references
\cite{abrasteg, erdelyiI, erdelyIII}.  It also relates to the relation
\begin{equation}
  \label{eq:eulerber}
  E_{n-1}(x)=2n^{-1}\bigl(B_n(x)- 2^nB_n(2^{-1}x)\bigr)
\end{equation}
between Euler and Bernoulli polynomials (\cite[23.1.27]{abrasteg},
\cite[1.14(7)]{erdelyiI}).
\begin{lem}\label{lem:viete}
  For $z\in\CC\setminus2\pi i\ZZ$, 
\begin{equation}
  \label{eq:id}
  \sum_{k=1}^\infty \frac{2^{-k}z}{\exp(2^{-k}z) + 1} = 1 - \frac{z}{\e^z - 1}\,.
\end{equation}
\end{lem}
Observe in passing that if we multiply both sides by $\e^z-1$ we obtain
an identity of entire functions.
\begin{proof}
For any integer $l\geq1$ we consider the identity
\begin{equation*}
  \e^z - 1 = (e^{z/2}+1)(\e^{z/4}+1)\cdots(\e^{z/2^l}+1)(\e^{z/2^l}-1),
\end{equation*}
and compute its logarithmic derivative. This gives, for $z$ not among the poles:
\begin{equation*}
  \frac{\e^z}{\e^z-1} = \sum_{k=1}^l \frac{\e^{z/2^k}2^{-k}}{\e^{z/2^k}+1} +
   \frac{\e^{z/2^l}}{2^l(\e^{z/2^l}-1)}. 
\end{equation*}
The limit for $l\to\infty$ gives:
\begin{equation*}
  1 + (\e^z-1)^{-1} = \sum_{k=1}^\infty \bigl(2^{-k} - \frac{2^{-k}}{\e^{z/2^k}+1}\bigr) + z^{-1},
\end{equation*}
which, after canceling out the $1$ and multiplying by $z$ gives Equation
\eqref{eq:id}. Another method starts from the identity $\sinh w = w\cosh(\frac w2)\cosh(\frac
w4)\cosh(\frac w8)\cdots$ and computes the logarithmic derivative at $w=\frac z2$.
\end{proof}
\begin{proof}[Proof of Proposition~\ref{prop:fm}]
We can now compute:
\begin{align*}
  f_m &= \sum_{l=0}^\infty \Bigl(1 - \frac{m/2^l}{\exp(m/2^l)-1}\Bigr)\\
      &= \sum_{l=0}^\infty  \sum_{k=1}^\infty \frac{2^{-l-k}m}{\exp(2^{-l-k}m)+1}\\
      &= \sum_{q=1}^\infty (\sum_{k=1}^q 1) \frac{2^{-q}m}{\exp(2^{-q}m)+1}\,.\qedhere
\end{align*}
\end{proof}
The statement and proof of Proposition~\ref{prop:fm} actually hold for any
complex $m$ which is not in $2\pi i\ZZ$.

In the next proposition we compute (the primitive of) the exponential
generating function for the $(e_m)_{m\geq0}$ sequence.
  \begin{prop}
    There holds (for $t\in \CC$):
    \begin{equation}\label{eq:gf}
      \sum_{m=0} e_m\frac{t^{m+1}}{(m+1)!} 
= (\e^t-1)\sum_{l=0}^\infty \Bigl(1 - \frac{2^{-l}t}{\e^{2^{-l}t}-1}\Bigr)
= (\e^t-1)\sum_{l=1}^\infty\frac{l2^{-l}t}{\e^{2^{-l}t}+1}\,.
    \end{equation}
  \end{prop}
  \begin{proof}
    Note that the quantity $E(t)$ defined by the left-hand side is an entire
    function, and indeed if we move the factor $\e^t-1$ to inside the sums we
    see that the two other expressions are, too.  Let us use as starting point Equation~\eqref{eq:ember2} (and recall $e_0=0$):
    \begin{align*}
      \sum_{m=1} e_m\frac{t^{m+1}}{(m+1)!}
&= -\sum_{m=1} \frac{t^{m+1}}{(m+1)!}\sum_{k=1}^m \frac{(m+1)!}{k!(m+1-k)!}B_k\frac{2^k}{2^k-1}
\\
&= -\sum_{k=1}^\infty \frac{B_k 2^k}{k!(2^k-1)}\sum_{m=k}^\infty \frac{t^{m+1}}{(m+1-k)!}
\\
&= -\sum_{k=1}^\infty \frac{B_k 2^k t^k}{k!(2^k-1)}(\e^t-1)
\\
&=-(\e^t-1)\sum_{k=1}^\infty B_k\frac{t^k}{k!}\sum_{l=0}^\infty 2^{-lp}
\\
&=-(\e^t-1)\sum_{l=0}^\infty \Bigl(\frac{2^{-l}t}{\e^{2^{-l}t}-1} - 1\Bigr).
    \end{align*}
    This gives the first equality from Equation~\eqref{eq:gf}.  The second one was
    already stated, in another variable, as Proposition~\ref{prop:fm}.
  \end{proof}
  \begin{rema}
    Hence, if we let $E(t)$ denote the left-hand side of
    Equation~\eqref{eq:gf}, $E(m)=(\e^m -1)f_m = \e^m f_m + O(\log m)$, $f_m =
    \e^{-m}E(m) + O(\log(m) \e^{-m})$, $e_m = \e^{-m}E_m + O(m^{-1})$.  This is
    very reminiscent of the estimate from \cite[Prop.\@ 1]{burnolosc} relating
    moments of some measures on the unit interval with their exponential
    generating function.  More generally, the link between the asymptotic of
    some sequences and of their (exponential or ordinary) generating functions
    is an old theme, already very present in \textsc{Knuth}'s Treatise
    \cite{knuth}, and which has been addressed by
    \textsc{Flajolet} and \textsc{Odlyzko} in a well-known paper
    \cite{flajoodlyz1990}.  See \cite[Part B: Complex
    asymptotics]{flajosedgebook} for an extensive introduction to this whole theme.
  \end{rema}
  \begin{rema}
    A form of the functional equation verified by $E(t)$ (which has neither
    constant term nor a linear term, as $e_0=0$) is
\begin{equation}\label{eq:funceq}
  E(2t)  = (\e^t+1){E(t)} + \e^{2t} -1-2t.
\end{equation}
Looking at~\eqref{eq:funceq} as an identity of formal power series, it implies
that $E$ has no constant term, but not that $E'(0)=0$.  Indeed, the
associated homogeneous equation has a one-dimensional
space of solutions generated by $\e^t-1$.  This corresponds to the fact that
the solutions of the recurrence relation Equation
\eqref{eq:emrec} are the sequences $(e_m + \lambda)_{m\geq0}$ for $\lambda$ an
arbitrary constant.
  \end{rema}

The $\phi$ function from Theorem~\ref{thm:em} fits into an old theme going
back to \textsc{Hardy} \cite{hardy1907} (see also \cite{keatread2000}).  \textsc{Balazard} et
al.\@ \cite{balamendsebb2005} have studied under some general conditions
series of the type $f_\alpha(x)=\sum_{n\in\ZZ}\alpha(x^{\theta^n})$ where
$\alpha$ is continuous on $\Iff01$, vanishes at $0$ and $1$ and has a power
series expansion $\alpha(x)=\sum_{k=1}^\infty a_kx^k$ (convergent for $0\leq
x<1$). They obtain the Fourier series of $f_\alpha$ as a function of
$\log_\theta\log x^{-1}$ in terms of values at $(\log\theta)^{-1}2\pi i \ZZ
\setminus\{0\}$ of the Gamma function and a Dirichlet series associated to
$\alpha$.  This is reminiscent of the computations done in section
\ref{sec:fourier} (we did not need to use the Poisson summation formula).

They consider in particular the case $\alpha(x)=x\prod_{n=0}^\infty (1 -
x^{2^n})$, which is related to the Thue-Morse sequence (see also about this
\cite{allouchecohen1985}).  It is of note that the function $\phi(t)$ ($t=\log_2(m)$)
from Theorem~\ref{thm:em} and Equation~\eqref{eq:phi} is related to
the logarithmic derivative of $m\mapsto \e^{m}\alpha(\e^{-m})$.  Indeed, we compute
straightforwardly:
\begin{equation*}
    m\ddm \log \prod_{n=0}^\infty (1 - \e^{-2^nm})
=
m\sum_{n=0}^\infty \frac{2^n\e^{-2^nm}}{1 - e^{-2^nm}}
=
\sum_{n=0}^\infty \frac{2^n m}{\e^{2^nm}-1}
=
\frac{m}{\e^m-1} - g_m\,.
\end{equation*}
Here $g_m$ is the quantity defined in Equation~\eqref{eq:gm}.

This computation motivates the consideration of the following infinite product:
\begin{equation}
  \label{eq:Psi}
  \Psi(z) = 2^{\frac{\Log_2(z)(\Log_2(z)+1)}2}
         \prod_{l=0}^\infty\frac{1 - \e^{-2^{-l}z}}{2^{-l}z} 
            \prod_{k=1}^\infty(1 - \e^{-2^kz}).
\end{equation}
The first infinite product is an entire function.  The second one is analytic in the half-plane $\Re z>0$.  So $\Psi$ is analytic there.  It verifies the functional equation:
\begin{equation*}
  \Psi(2z) = 2^{\frac{(\Log_2(z)+1)(\Log_2(z)+2)}2-\frac{\Log_2(z)(\Log_2(z)+1)}2}
              \Psi(z)\frac{(1-\e^{-2z})/(2z)}{1 - \e^{-2z}}=\Psi(z)
\end{equation*}
A straightforward computation gives the logarithmic derivative with respect to
the variable $\Log z$:
\begin{equation*}
  z\ddz \log \Psi(z) = \Log_2(z) + \frac12 
     + \sum_{l=0}^\infty \Bigl(\frac{2^{-l}z}{\e^{2^{-l}z}-1} - 1\Bigr)
     + \sum_{k=1}^\infty\frac{2^k z}{\e^{2^kz} - 1}.
\end{equation*}
Comparison with Equation~\eqref{eq:phi} shows that the above is exactly
$-\phi(\Log_2 z)$.  And indeed, the image of the right half-plane under
$z\mapsto\Log_2z$ is the horizontal strip where we have defined $\phi$ as an
analytic function.

Whether this relation between $\Psi$ and $\phi$ is indicative of some closer
connection between the sequence $(e_m)$ and the Thue-Morse sequence
\cite{allouchecohen1985,balamendsebb2005} is currently unknown to the author.

\noindent\textbf{Acknowledgements.} Thanks to Jean-Paul \textsc{Allouche} for
drawing the author attention to \href{https://oeis.org/A372422}{A372422} and
to the work
by \textsc{Knuth} and \textsc{Prodinger}.

\providecommand\bibcommenthead{}
\def\blocation#1{\unskip} \footnotesize

%



\bigskip
\noindent\hfill\vtop{%
\hsize 5cm
\parskip 0pt plus 1pt minus 1pt
\footnotesize
\flushleft
\obeylines\strut%
  Université de Lille,
  Faculté des Sciences et technologies,
  Département de mathématiques,
  Cité Scientifique,
  F-59655 Villeneuve d'Ascq cedex,
  France
\medskip%
\strut jean-francois.burnol@univ-lille.fr
}%

\end{document}